\documentclass[a4paper,10pt]{article}
\usepackage[utf8]{inputenc}
\usepackage{setspace} 
\usepackage{mathrsfs}

\usepackage{comment}

\usepackage{amsmath,amssymb,amsthm,ulem}
\usepackage{amsfonts, 
mathtools,enumerate,enumitem,esint}
\usepackage{color}
\usepackage{hyperref}
\usepackage{verbatim} 

\theoremstyle{plain}
\newtheorem{theorem}{Theorem}[section]
\newtheorem{proposition}{Proposition}[section]
\newtheorem{lemma}{Lemma}[section]
\newtheorem{corollary}{Corollary}[section]
\newtheorem{definition}{Definition}[section]

\newtheorem{remark}{Remark}[section]
\numberwithin{equation}{section}

\def\di{\mathrm{div}\,}
\def\dist{\mathrm{dist}\,}

\def\supp{\mathrm{supp}\,}
\def\bN{\mathbb{N}}
\def\bR{\mathbb{R}}
\def\bT{\mathbb{T}}
\def\bZ{\mathbb{Z}}
\def\cH{\mathcal{H}}

\def\cL{\mathcal{L}}

\def\div{\mathrm{div}\,}

\def\eps{\varepsilon}
\def\vfi{\varphi}
\def\dd{\,\mathrm{d}}

\def\lb{\langle} \def\rb{\rangle}

\title{The heat equation with the dynamic boundary condition as a singular limit of problems degenerating at the boundary} 
\author{Yoshikazu Giga$^*$, Micha{\l} {\L}asica$^{*,\dagger}$, Piotr Rybka$^\ddagger$ \\
\small{$^*$ The University of Tokyo, Graduate School of Mathematical Sciences} \\
\small{$^\dagger$ Institute of Mathematics of the Polish Academy of Sciences} \\
\small{$^\ddagger$ University of Warsaw, Institute of Applied Mathematics and Mechanics}}

\begin{document}

\maketitle

\begin{abstract}
We derive the dynamic boundary condition for the heat equation as a limit of boundary layer problems.
We study convergence of their  weak and strong solutions 
as the width of the layer tends to zero. We also discuss $\Gamma$-convergence of the functionals generating these flows. Our analysis of strong solutions depends on a new version of the Reilly identity.
\end{abstract}



\section{Introduction}
We consider a singular limit problem of a simple heat equation with conductivity coefficient $a_\varepsilon$
\begin{equation} \label{eq2}
	b_\varepsilon u_t^\varepsilon = \operatorname{div}(a_\varepsilon\nabla u^\varepsilon) \quad\text{in}\quad
	\Omega\times(0,T) =: \Omega_T
\end{equation}
in an $N$-dimensional domain $\Omega$, which is often called a concentrating capacity problem.
Namely, the heat capacity $b_\varepsilon$ is very high in an $\varepsilon$-neighborhood of the boundary. Here is its 
form
\begin{equation} \label{def-be} 
	b_\varepsilon(x) = \chi_{\Omega_\varepsilon} + \phi(\varepsilon) \chi_{\Omega\backslash\Omega_\varepsilon}
\end{equation}
where $\Omega_\varepsilon=\left\{ x\in\Omega \mid d(x)\geq\varepsilon \right\}$ and $\lim_{\varepsilon\to0} \phi(\varepsilon)=\infty$. 
The symbol $\chi_K$ denotes the characteristic function of a set $K$ and $d(x)$ stands for the distance of $x$ to the boundary $\partial\Omega$ of $\Omega$, i.e., $d(x)=\operatorname{dist}(x,\partial\Omega)$.
The conductivity coefficient $a_\varepsilon$ can be degenerate at $\partial\Omega$.
A typical choice of $a_\varepsilon$ is of the form
\begin{equation} \label{dea} 
	a(x) := \min \left\{ \frac{d(x)}{\varepsilon}, 1 \right\}.\end{equation}
We consider a limit of  solutions to \eqref{eq2} as $\varepsilon$ tends to zero under proper boundary conditions and initial data.
In order to make our presentation clear and to avoid non essential technical difficulties we consider $\Omega$ as a flat cylinder of the form
\begin{equation}\label{def-om}
\Omega = \mathbb{T}^{N-1} \times (0, 1),
\end{equation}
where $\mathbb{T}^{N-1}=\mathbb{R}^{N-1}/\mathbb{Z}^{N-1}$ is a flat torus.

It turns out the singular limit $u$ of $u^\varepsilon$ as $\varepsilon\to0$ depends on the value
\[
\lim_{\eps\to 0} \phi(\varepsilon) \varepsilon = \kappa \in [0,\infty],
\]
which is assumed to exist. If $\kappa\in(0,\infty)$, then we perform the limit passage under suitable assumptions on initial data of $u^\varepsilon$.
We show that the singular limit $u$ 
solves the heat equation with dynamic boundary condition
\begin{eqnarray} \label{eq1}
\left\{
\begin{array}{c}
	u_t = \Delta u \quad\text{in}\quad \Omega_T \vspace{4pt}\\
	\displaystyle\gamma(u)_t + \frac{1}{\kappa} \frac{\partial u}{\partial\nu} = 0 \quad\text{on}\quad \partial\Omega\times(0,T), \vspace{4pt}\\
	u(0,x) = u_0(x).
\end{array}
\right.
\end{eqnarray}
Here, $\gamma$ denotes the trace operator and $\nu$ denotes the outer unit normal to $\Omega$. If $\kappa=\infty$, the limit $u$ must solve the homogeneous Dirichlet problem for the heat equation. In the case $\kappa=0$, the limit $u$ must solve the corresponding homogeneous Neumann problem.

This type of result has been already established by Colli and Rodrigues \cite{colli} for zero initial data but for a non-zero force term. Although they consider general elliptic operator for the diffusion term, their assumption does not allow the degeneracy of $a_\varepsilon$ at $\partial\Omega$. In other words our $a_\varepsilon$ in \eqref{dea} is excluded. Recently, results very close in spirit to ours  were presented in \cite{eng-pipe}. The authors study there the problem of shrinking the thickness of the pipe, through which a heat conducting fluid flows. 

In order to derive our convergence results, we will exploit 
the fact that both flows \eqref{eq1} and \eqref{eq2} are gradient flows. 
Our goal is to look at this problem from two different angles. We first consider weak solutions. Namely, we mean $u^\varepsilon$ in $L^2(0,T,H^1)$ while $u_t^\varepsilon$ is in $L^2\left(0,T,(H^1)^*\right)$.
In this setting, the initial data $u_0^\varepsilon$ of $u^\varepsilon$ is allowed to be $L^2(\Omega)$ for the existence result.
Since the trace is not well defined for $L^2$-functions, we introduce an 
averaging operator $m_\varepsilon$ acting over the  $\varepsilon$-neighborhood of $\partial\Omega$ to find a class of well prepared initial data. More specifically, we say that the family  $u_0^\varepsilon$ is well prepared if there exists $(u_0,w_0)\in L^2(\Omega) \times L^2(\partial\Omega)$ such that
\[
	\sup_{\varepsilon>0} \int_\Omega b_\varepsilon (u_0^\varepsilon)^2\, dx < \infty
	\quad\text{and}\quad (u_0^\varepsilon, m_\varepsilon u_0^\varepsilon) \to (u_0, w_0) 
	\quad\text{in}\quad L^2(\Omega) \times L^2(\partial\Omega).
\]
For such 
initial data $u_0^\varepsilon$, we prove that the corresponding weak solutions $u^\varepsilon$ to \eqref{eq2} converge weakly to a weak solution of \eqref{eq1}. This  is stated in Theorem \ref{thm_conv_dbc}, when $\kappa\in(0,\infty)$. We regard \eqref{eq2} as a gradient flow $b_\varepsilon u_t^\varepsilon\in-\partial E_\varepsilon(u)$ with $E_\varepsilon(u)=\frac12\int_\Omega a_\varepsilon \left|\nabla u^\varepsilon\right|^2\, dx$ but the a priori estimate provided by taking inner product of this equation with $u^\varepsilon$ is sufficient for the limit passage. There is no need to address the convergence of the corresponding variational functionals like $E_\varepsilon$, but we will do this for the sake of completeness of analysis. Interestingly, we can simultaneously consider regular problem (when $\inf a_\varepsilon>0$) and degenerate ones, provided that $a_\varepsilon$ vanish in a controllable manner at the boundary, see \eqref{asa_nondeg}.

Our convergence result, see Theorem \ref{thm_conv_dbc}, holds for natural boundary conditions (we do not discuss their meaning here), when $\kappa\in (0,\infty)$. 
We separately treat the easier case of $\kappa=0$ leading to the homogeneous Neumann conditions in the limit or $\kappa=\infty$ yielding the homogeneous Dirichlet conditions for $u$, which are well-understood for weak solutions.  We stress that these results are based on the analysis of behavior of averages of solutions of \eqref{eq2} over the boundary layer.
 
We now turn our discussion to strong solutions, $u$, of (\ref{eq2}). By definition they are  
in $L^2(0,T,H_{loc}^2)$ while $u_t\in L^2(0,T,L^2)$.
Although we eventually show in Theorem \ref{thmain} the convergence of solutions of equation \eqref{eq2} to \eqref{eq1} for all $\kappa\in(1,\infty)$ in case of well-prepared data (in a stronger topology), there are two main differences from Theorem \ref{thm_conv_dbc}. The first problem is the necessity to study the behavior of the traces and their convergence, when $a^\eps$ is degenerate. In fact we consider only the case of $a^\eps$ given by (\ref{dea}). 

For this purpose, we have to control derivatives of $u^\varepsilon$ belonging locally  to $H^2$. The need to estimate the second derivatives in terms of $\operatorname{div}(a_\varepsilon\nabla u^\varepsilon)$ leads us to consider the Reilly-type, which is the content of Theorem \ref{lem1.8}.

The second problem is that one has to study the $\Gamma$-convergence of functionals $E_\varepsilon$. The peculiarity of our problem is that (\ref{eq2}) is the gradient flow of $E_\eps$ with respect to the inner product, which changes with $\eps,$
$$
\lb u, v\rb_\eps = \int_\Omega  b_\varepsilon u v  \,dx,
$$
where $b_\eps$ is given in (\ref{def-be}).
This inner product defines a very useful Hilbert space,  $L^2_\eps(\Omega)$, namely,
\begin{equation}\label{df-Le}
L^2_\eps(\Omega) = (L^2(\Omega), \lb \cdot, \cdot\rb_\eps).
\end{equation}

This $L^2_\eps(\Omega)$ space changing with $\eps$ is a reason why the usual notion of $\Gamma$-convergence yields incorrect results. We have to use the $\Gamma$-convergence with respect to an immersion $\iota_\eps$. This permits us to deal with the situation, when the limiting space is bigger then the  base of $E_\eps,$ see Definition \ref{dGC} and Lemma \ref{lem1.10}. The modified $\Gamma$-convergence is used to identify properties of the limit of solution. At the same time the convergence was enforced with the 
help of the
Energy-Dissipation-Balance originally due to de Giorgi with further more recent extensions by Serfaty, \cite{sylvia}, or Mielke \cite{M}.

Let us comment on the Reilly-type identity we derive in Section \ref{s-Re}. Originally the geometric context was essential. Here, it is not the  case. To some extent, our derivation of this identity is similar to the singularity analysis of solution to the Laplace equations in polygonal domains, see e.g. \cite[Theorem 2.2.1]{grisvard}.  We believe that the derivation of Reilly-type identity is of independent interest. That is why a whole section of this paper is devoted solely to this topic. We expect that this result could be generalized to any region with the $C^2$-boundary.

We are interested in deriving the dynamic boundary conditions  (\ref{eq1}$_2$) for the heat equation. This is closely related to the topic studied by Savar\'e-Visintin,  see \cite{SV}. They suggested blowing up the boundary in the  spirit of the Gamma-limit of variational functionals. They had in mind the transition problem as the origin for the limit process. The authors in \cite{SV} studied the anisotropic elliptic problem, where the direction perpendicular to the transmission surface dominates. 

The paper \cite{SV} is quite different from ours, because when a small parameter goes to zero, then the ellipticity matrix becomes large in the direction perpendicular to the surface dividing the given region into two with different properties, what is origin of the transmission. Subsequently, the limit passage is performed. Another passage studied in that paper is the case when one of the regions tends to a manifold and in the limit we obtain a second order parabolic eq. coupled to another parabolic problem in $\bR^N$.

The setting and the results presented by Colli and Rodrigues, \cite{colli}, are  much closer to ours. Our heat equation is a special case of their transmission problem.  Colli-Rodrigues obtain a range of similar results depending on the value of the parameter, which we call $\kappa$. The main difference is that our problems degenerate at the boundary, while this is not the case in \cite{colli}.

In a recent paper, \cite{eng-pipe}, Ljulj {\it et al.} consider a heat conducting fluid in a pipe of finite thickness. Depending on the properties of the pipe the authors obtain a number of boundary conditions for the heat equation in the fluid, when the thickness of the wall goes to zero. The analysis of \cite{eng-pipe} is based on the multiple scale convergence. In this case the thin wall is blown to a fixed region.

At last we  should comment on the boundary conditions. In case of the strong solutions and $a_\eps$ given by (\ref{dea}), 
there is no need to complement (\ref{eq2}) with boundary data. It turns out that the regularity makes them implicit. More precisely, we prove in Lemma \ref{lem1.1} that if $E_\eps(u)$ is finite and $\di (a_\varepsilon \nabla u)$ is in $L^2(\Omega)$, then the trace $\gamma(a_\eps \frac{\partial u}{\partial x_N})$ exists and it vanishes as an $L^2$-function.  

We cannot claim that much for weak solutions for the lack of necessary regularity. In case of degenerating weights $a_\eps$ we formally impose the natural boundary conditions,
$$
\gamma(a_\eps \frac{\partial u}{\partial x_N} )=0.
$$
However, they should be understood as a justification for the lack of boundary integrals in the definition of weak solutions, see Definition \ref{weak_bl_def}. 


Here is the plan of our paper. Since we deal a lot with averages over the boundary layer and the relationship between them and the  traces we devote Section \ref{sec2} to these topics. In Section 3 we establish the convergence of weak solutions for all the cases $\kappa\in[0,\infty]$. We derive the Reilly-type identity in Section 4. Section 5 is devoted to convergence of strong solution, that exploit the energy-dissipation balance and the Reilly identity. In this section  we consider only the regular  case $\kappa\in(0,\infty)$. The singular cases, $\kappa\in\{0,\infty\}$, which do not need the Reilly identity are considered in the last Section 6.

\section{Preliminaries: facts on traces and averages}\label{sec2}
In this section we gather facts on traces for both types of solutions, which we consider here. The weight $b_\eps$, which creates the boundary layer, leads us to consider averages over $\Omega\setminus \Omega_\eps$. The following average, defined for each $\eps \in ]0,1[$, plays the 
crucial role in our considerations. Namely, we set $m_\eps\colon L^p(\Omega) \to L^p(\partial\Omega)$, $p\in[1,\infty)$ by formula
\begin{equation}\label{def-we}
(m_\eps u)(x') = 
\left\{
\begin{array}{ll}
\frac 1\varepsilon\int_0^\varepsilon u(x',s)\,ds  & \hbox{if } x_N =0,\\
\frac 1\varepsilon\int_{1-\varepsilon}^1 u(x',s)\,ds  & \hbox{if } x_N =1 ,
\end{array}
\right.    
\end{equation}
where $(x', x_N)\in \partial\Omega$.

We also discuss traces at a distance $\eps$ from the boundary. For this purpose we use the trace operator $\gamma^\eps: H^1(\Omega_\eps)\to L^2(\partial\Omega)$, where we  identify $L^2(\partial\Omega)$ with $L^2(\partial\Omega_\eps)$. We notice that for $\eps\in(0, \frac12)$, we have $\partial \Omega_\eps =\bT^{N-1}\times\{\eps, 1-\eps\} $.

We have to specify the energy functionals. Namely, for $\eps\in (0,\frac{1}2)$, we define
\begin{equation}\label{defE} 
E_\eps(u) =\left\{ 
\begin{array}{ll}
\frac12\int_\Omega a_\eps|\nabla u|^2\,dx  & u\in H^{1}_{loc}(\Omega), \\
 +\infty   & u \in L^2(\Omega)\setminus  H^{1}_{loc}(\Omega) .
\end{array}
\right.
\end{equation}
A critical point in this definition is the choice of $a_\eps$. We always assume that
\begin{equation}\label{asa}
a_\eps\in C(\overline\Omega), 
\qquad \{x\in \overline\Omega: a_\eps =0\}\subset\partial\Omega,\qquad \lim_{\eps\to0} a_\eps = 1\quad\hbox{locally uniformly in } \Omega. 
\end{equation}
We  consider  $a_\eps(x)$ depending only on the distance from $x$ to the boundary of $\partial\Omega$,
\begin{equation}\label{a-dis}
a_\eps(x) = a_\eps(d(x)),
\end{equation}
where $d(x)$ denotes the distance of $x$ from the boundary $\partial\Omega$, i.e., $d(x)=\operatorname{dist}(x,\partial\Omega)$; here we abuse notation using $a_\varepsilon$ to represent also a function of $d$. 

To establish our results, we will also require  a
non-degeneracy condition, which we state as  follows, 
\begin{equation}\label{asa_nondeg}
a_\eps \geq \underline{a} > 0 \quad \text{on } \Omega_\eps, \qquad  
\frac{1}{\eps}\int_0^\eps \left(\int_s^\eps \frac{1}{a_\eps(\sigma)}\dd \sigma\right)^p \dd s 
\to 0 \quad \text{as } \eps \to 0^+,
\end{equation}
where $p =\frac{1}{2}$ or $p=1$.
In order to simplify the notation, we will assume that $\underline{a} =1$.

In particular we may choose $a_\eps \equiv 1.$ This possibility will be considered only for weak solutions.
In general we have to know how fast $a_\eps$ goes to zero, when the argument approaches the boundary, i.e. we need (\ref{asa_nondeg}).


Our statements on traces depend on the notion of a solution we consider. When we discuss the strong solutions we expect the following bound on the data,
\begin{equation}\label{s-trace1}
\sup_{\eps>0} E_\eps( u^\eps_0)< \infty.
\end{equation}
On the other hand,  for  weak solutions we would rather expect
\begin{equation}\label{s-trace2}
\sup_{\eps>0} \int_\Omega b_\eps (u_0^\eps)^2\, dx<\infty.
\end{equation}
As we shall see in Theorem \ref{thm_conv_dbc}, estimate (\ref{s-trace2}) will imply an integrated in time version of  (\ref{s-trace1}) see (\ref{energy_eq_bl}). It turns out that this is sufficient.

After these preliminary definitions, we make our first observation.

\begin{lemma}\label{traceU} Let us suppose that $u^\eps\in L^2(\Omega),$ $\eps>0$, and
$$
\sup_{\eps>0} \| u^\eps\|_{L^2},\quad \sup_{\eps>0} E_\eps(u^\eps)<\infty.
$$
Then, \\
(0) there is a subsequence (without relabeling), which converges when $\eps\to0$ to $u$ weakly in $L^2(\Omega)$ and weakly in $H^{1}(\Omega_\delta)$ for all $\delta>0$. Moreover,\\
(1) If $u$ is a weak limit of  $u^\varepsilon$  in the sense described in (0) above, then
$$
\gamma (u) = \lim_{\varepsilon\to 0} \gamma^\varepsilon(u^\varepsilon),
$$
where $\gamma^\varepsilon$ is defined above  and the limit is taken in the $L^2(\partial\Omega)$-norm.\\
(2) If in addition $a_\eps$ is defined by (\ref{dea}) and
$$
\sup_{\eps>0} \int_{\Omega_\eps} |\hbox{\rm div}\, (a_\eps \nabla u^\eps)|^2\, dx < \infty,
$$
then there is another subsequence (not relabelled) converging to $u $ weakly in $H^{2}(\Omega_\delta)$ for all fixed $\delta>\eps$. The limit is in $H^2(\Omega)$ and 
$$
\gamma (\frac{\partial u}{\partial x_N}) = \lim_{\varepsilon\to 0} \gamma^\varepsilon(\frac{\partial u^\varepsilon}{\partial x_N})\qquad\hbox{in }L^2(\partial\Omega).
$$
\end{lemma}
\begin{remark}
In part (2), we can take any other partial derivative $\frac{\partial u}{\partial x_i}$, $i=1,\ldots, N-1.$
\end{remark}
\begin{proof} {\it Step (a).} The existence of a subsequence convergent in $L^2(\Omega)$ is automatic. The existence of a further subsequence convergent in $H^{1}_{loc}(\Omega)$ follows from the diagonal process. The regularity of the limit is obvious. Details are left out.

{\it Step (b).} The boundary of $\Omega$ has two diffeomorphic components. It is sufficient to consider one of them and we choose $\bT^{N-1} \times \{0\}$.
Now, we may introduce  a family of functions $v^\eps$ defined by $v^\eps = u^\eps \psi$, where $\psi\in C^\infty(\bR)$ is a cut-off function, such that $\psi|_{[0,\frac 14]} = 1$ and $\psi|_{[\frac 12, 1]} = 0$. 

Now, the family $v^\eps$ is weakly convergent in $L^2(\Omega)$ to $v$, $v^\eps \equiv 0$ on $\bT^{N-1} \times (\frac12, 1)$ and $\int_{\Omega_\eps} |\nabla v^\eps|^2\, dx < \infty$. Then, the family $v^\eps(\cdot+\eps)$ is uniformly bounded in $H^{1}(\Omega)$ for $\eps\in (0,\frac12)$. Since the trace $\gamma: H^{1}(\Omega) \to H^{1/2}(\partial\Omega)$ is bounded, we deduce from the Sobolev embedding theorem that   the family $\gamma(v^\eps(\cdot + \eps))$ is precompact in $L^2(\partial\Omega)$. So, possibly after extracting another subsequence, we conclude that $\gamma(v^\eps(\cdot + \eps))$ converges in $L^2(\partial \Omega)$ to some $w$.

{\it Step (c)} We first show that $\frac{\partial v^\eps}{\partial x_N} $ converges weakly to $\frac{\partial v}{\partial x_N}$. More precisely, for any $\vfi\in L^2(\Omega)$, we have
$$
\lim_{\eps\to 0} \int_{\Omega_\eps} \frac{\partial v^\eps}{\partial x_N} \vfi \,dx =
\int_{\Omega} \frac{\partial v}{\partial x_N} \vfi \,dx.
$$
Indeed,  let us fix $\eta>0.$ Since $v\in H^1(\Omega)$,
we will consider only those $\delta>0$ for which 
\begin{equation}\label{de-ta}
\int_{\Omega\setminus \Omega_\delta} \left|\frac{\partial v}{\partial x_N} \right|^2 \,dx \le \eta.
\end{equation}
Then, for any  $\eps<\delta$, we have
\begin{eqnarray*}
&&\int_{\Omega_\eps} \frac{\partial v^\eps}{\partial x_N} \vfi \,dx 
- \int_{\Omega} \frac{\partial v}{\partial x_N} \vfi \,dx 
\\ &=& 
\int_{\Omega_\eps\setminus\Omega_\delta} \frac{\partial v^\eps}{\partial x_N} \vfi \,dx +
\int_{\Omega_\delta}
\left( \frac{\partial v^\eps}{\partial x_N}- \frac{\partial v}{\partial x_N}\right)
\vfi \,dx
- \int_{\Omega\setminus\Omega_\delta} \frac{\partial v}{\partial x_N} \vfi \,dx 
\\ &= & I_1 + I_2 + I_3.
\end{eqnarray*}
Our assumption on the uniform boundedness of energies and the Cauchy-Schwarz inequality imply that
$$
| I_1| \le \sqrt 2 \sqrt{ E_\eps(v^\eps)} \| \vfi\|_{L^2(\Omega\setminus\Omega_\delta)}\le \eta,
$$
after restricting $\delta$ one more time. By the same token, estimate (\ref{de-ta}) yields $|I_3|\le \sqrt\eta \| \vfi \|_{L^2(\Omega\setminus\Omega_\delta)}$. Finally, since $v^\eps$ converges weakly to $v$ locally in $H^1(\Omega)$, so in particular in $H^1(\Omega_\delta)$, we deduce that $I_2$ goes to zero as $\eps \to 0$.

{\it Step (d)} Since $\gamma(v^\eps(\cdot + \eps)) = \gamma^\eps(v^\eps)$, the formula
$$
\gamma^\eps(v^\eps(x)) = - \int_\eps^1 \frac\partial{\partial x
_N} v^\eps(x,s)\, ds
$$
shows that weak convergence of $\frac{\partial v^\eps}{\partial x_N}$ implies that $\gamma^\eps(v^\eps)$ converges weakly to $\gamma(v).$ Combining this with Step (b) yields $w = \gamma(v)$.

{\it Step (e)} Under the boundedness assumption of Lemma \ref{traceU}(2), we invoke the Reilly-type identity, Theorem \ref{lem1.8}, to deduce that for all $\delta>0$
$$
\| u^\eps \|_{H^2(\Omega_\delta)} \le M<\infty,
$$
where $M$ does not depend on $\delta$. Our claim follows, in particular $u\in H^2(\Omega).$

{\it Step (f)} We will simultaneously show  (1) and (2). For this purpose we use steps (b) and (d). We  set
$$
v^\varepsilon(x', x_N ) = u^\varepsilon(x', x_N)\eta(x_N), \qquad(\hbox{resp. } v^\varepsilon(x', x_N ) = \frac{\partial u^\eps}{\partial x_N}(x', x_N)\eta(x_N)),
$$
where $\eta$ is a cut-off function with the support in $(0,\frac14)$ and equal to one on $(0,\frac18).$
Then $v^\varepsilon$ satisfies conditions assumed in step (b). Hence, step (d) yields (1) and (2).
\end{proof}

In case of weak solutions we want to replace uniform boundednes of energies with boundedness of the time integral of energies. However, this relaxed assumption  gives us only the  weak convergence of traces.

\begin{lemma}\label{traceU-w} Let us suppose that there exist $0\le t_0< t_1<\infty$ and $u^\eps\in L^2(\Omega\times (t_0,t_1))$ such that
$$
\sup_{\eps>0} \| u^\eps\|_{L^2(\Omega\times (t_0,t_1))},\quad \sup_{\eps>0} \int_{t_0}^{t_1} E_\eps(u^\eps(s))\,ds <\infty.
$$
Then, \\
(0) there exists $u\in L^2(t_0,t_1; H^1(\Omega))$ and a subsequence (without relabeling), $u^\eps$, which converges when $\eps\to0$ to $u$ weakly in $L^2(t_0,t_1; H^1(\Omega_\delta))$ for all $\delta>0$. Moreover,\\
(1) If $u$ is this  weak limit of  $u^\varepsilon$, then
$$
\gamma (u) = \lim_{\varepsilon\to 0} \gamma^\varepsilon(u^\varepsilon),
$$
where $\gamma^\varepsilon$ is defined above  and the limit is in the weak $L^2(\partial\Omega\times(t_0,t_1))$-topology.
\end{lemma}
\begin{proof} {\it Step (a).} We proceed as in the previous lemma to show (0).

{\it Step (b).} Similarly to the previous lemma, we consider just one component of  $\partial\Omega$, we choose $\bT^{N-1} \times \{0\}$.
We show that $\frac{\partial v^\eps}{\partial x_N} $ converges weakly to $\frac{\partial v}{\partial x_N}$ in $L^2(t_0, t_1: H^1(\Omega))$. More precisely, we have
$$
\lim_{\eps\to 0}\int_{t_0}^{t_1} \int_{\Omega_\eps} \frac{\partial v^\eps}{\partial x_N} \vfi \,dxdt =
\int_{t_0}^{t_1}\int_{\Omega} \frac{\partial v}{\partial x_N} \vfi \,dxdt.
$$
Indeed,  let us fix $\eta>0.$ We will consider only those $\delta>0$ for which 
\begin{equation}\label{de-tai}
\int_{t_0}^{t_1}\int_{\Omega\setminus \Omega_\delta} \left|\frac{\partial v}{\partial x_N} \right|^2 \,dxdt \le \eta.
\end{equation}
Then, for any $\vfi \in L^2(\Omega\times(t_0,t_1))$ and $\eps<\delta$ we have
\begin{eqnarray*}
&&\int_{t_0}^{t_1}\int_{\Omega_\eps} \frac{\partial v^\eps}{\partial x_N} \vfi \,dx dt
- \int_{t_0}^{t_1}\int_{\Omega} \frac{\partial v}{\partial x_N} \vfi \,dxdt 
\\ &=& 
\int_{t_0}^{t_1}\int_{\Omega_\eps\setminus\Omega_\delta} \frac{\partial v^\eps}{\partial x_N} \vfi \,dx dt+
\int_{t_0}^{t_1}\int_{\Omega_\delta}
\left( \frac{\partial v^\eps}{\partial x_N}- \frac{\partial v^\eps}{\partial x_N}\right)
\vfi \,dxdt
-\int_{t_0}^{t_1} \int_{\Omega_\eps\setminus\Omega_\delta} \frac{\partial v}{\partial x_N} \vfi \,dxdt 
\\ &= & I_1 + I_2 + I_3.
\end{eqnarray*}
Our assumption on the uniform boundedness of energies and the Cauchy-Schwarz inequality imply that
$$
| I_1| \le \sqrt 2 \left(\int_{t_0}^{t_1} E_\eps(v^\eps) \right)^{1/2}
\| \vfi\|_{L^2((\Omega\setminus\Omega_\delta) \times(t_0, t_1))}\le \eta,
$$
after restricting $\delta$ one more time. By the same token, estimate (\ref{de-tai}) yields $|I_3|\le \sqrt\eta \| \vfi \|_{L^2((\Omega\setminus\Omega_\delta) \times(t_0, t_1))}$. Finally, since $v^\eps$ converges weakly to $v$  in $L^2(t_0,t_1;H^1(\Omega_\delta))$, for all $\delta>0$, we deduce that $I_2$ goes to zero when $\eps \to 0$.

{\it Step (d)} Since $\gamma(v^\eps(\cdot + \eps)) = \gamma^\eps(v^\eps)$, the formula
$$
\gamma^\eps(v^\eps(x)) = - \int_\eps^1 \frac\partial{\partial x
_N} v^\eps(x,s)\, ds
$$
shows that the weak convergence of $\frac{\partial v^\eps}{\partial x_N}$ implies that $\gamma^\eps(v^\eps)$ converges weakly to $\gamma(v)$ in   $L^2(t_0,t_1;L^2(\partial\Omega_\delta))$
\end{proof}

We have just shown that traces $ \gamma^\varepsilon(u^\varepsilon)$ converge to $\gamma(u)$. However, we will rather deal with averages, defined in (\ref{def-we}) not $ \gamma^\varepsilon(u^\varepsilon)$. At the same time we may choose a convenient topology. We also notice that our result depends on the coefficient $a_\eps$.
Our main observation is the following, 

\begin{lemma}\label{lem11} Let us suppose that
$$
\sup_{\eps>0} E_\eps (u^\eps) \le M<\infty.
$$
(a) If  we take $p= 1$ in (\ref{asa_nondeg}),
then
$$
 \lim_{\varepsilon\to 0}\gamma^\varepsilon(u^\varepsilon) =
 \lim_{\varepsilon\to 0} m_\varepsilon(u^\eps)\qquad\hbox{ in }L^2(\partial\Omega).
$$
(b) If we take $p= 1/2$ in (\ref{asa_nondeg}),
then
$$
 \lim_{\varepsilon\to 0}\gamma^\varepsilon(u^\varepsilon) =
 \lim_{\varepsilon\to 0} m_\varepsilon(u^\eps)\qquad\hbox{ in }L^1(\partial\Omega).
$$ 
\end{lemma}
\begin{proof} Part (a). Since the components of the boundary of $\Omega$ are diffeomorphic, we present the argument only for $\bT^{N-1}\times\{0\}$.

We shall estimate
\begin{eqnarray*}
\|\gamma^\varepsilon(u^\varepsilon) - m_\eps(u^\eps)\|^2_{L^2}&=&\int_{\partial\Omega} | \gamma^\varepsilon(u^\varepsilon) - \frac1\varepsilon \int_0^\varepsilon u^\varepsilon \,dx_N|^2\, d\cH^{N-1}\\ &=&
\int_{\partial\Omega} 
| \frac1\varepsilon \int_0^\varepsilon  (u^\varepsilon (x', \varepsilon) - 
u^\varepsilon(x', s) \,ds|^2\, d\cH^{N-1} \\&=&\frac1{\varepsilon^2}
\int_{\partial\Omega} 
|  \int_0^\varepsilon  (u^\varepsilon (x', \varepsilon) - 
u^\varepsilon(x', s) \,ds|^2\, d\cH^{N-1} \\&=&
\frac 1{\varepsilon^2} \int_{\partial\Omega}
\left| \int_0^\varepsilon\int_s^\varepsilon \frac{\partial u^\varepsilon}{\partial x_N}(x', y) \,dyds \right|^2 \,d\cH^{N-1}.
\end{eqnarray*}
Next, we apply twice the Cauchy-Schwarz inequality to obtain,
\begin{eqnarray*}
\|\gamma^\varepsilon(u^\varepsilon) - m_\eps(u^\eps)\|^2_{L^2} &\le &
\frac \varepsilon{\varepsilon^2} \int_{\partial\Omega}
\int_0^\varepsilon\left|\int_s^\varepsilon  \frac{\partial u^\varepsilon}{\partial x_N}(x', y) \,dy \right|^2 \,dsd\cH^{N-1}\\
&= & 
\frac 1{\varepsilon} \int_{\partial\Omega}
\int_0^\varepsilon \left|\int_s^\varepsilon \frac 
{a_\eps^{1/2}(y)}{a_\eps^{1/2}(y)}
\frac{\partial u^\varepsilon}{\partial x_N}(x', y)  \right|^2 \,dydsd\cH^{N-1}\\&\le&
\frac1\eps
 \int_{\partial\Omega}
\int_0^\varepsilon
\int_s^\eps \frac1{a_\eps(y)}\, dyds
\int_0^\varepsilon 
a_\eps(y)
\left| \frac{\partial u^\varepsilon}{\partial x_N}(x', y)  \right|^2 \,dsdyd\cH^{N-1}.
\end{eqnarray*}
Since we take $p=1$ in (\ref{asa_nondeg}),
%
we obtain
\begin{equation*}
\|\gamma^\varepsilon(u^\varepsilon) - m_\eps(u^\eps)\|_{L^2}^2 
\le M \frac1\eps
 \int_{\partial\Omega}
\int_0^\varepsilon
\int_s^\eps \frac1{a_\eps(y)}\, dyds \stackrel {\eps\to 0}\longrightarrow 0. 
\end{equation*}
Computations for part (b) follow the same lines,
$$
\|\gamma^\varepsilon(u^\varepsilon) - m_\eps(u^\eps)\|_{L^1}
\le \frac 1\eps \int_0^\eps  \left(\int_s^\eps \frac {d\sigma}{a_\eps(\sigma)}
\right)^{1/2}
\left( \int_{\partial\Omega}\int_s^\eps a_\eps(\sigma) 
\left|\frac{\partial u^\eps}{\partial x_N}\right|^2 d\,\sigma\right)^{1/2} \,ds.
$$
Our claim follows.
\end{proof}
\begin{remark}\label{pot-a}
A simple choice of $a_\eps$ is given by 
$$a_\varepsilon (d) = \min\left(1, d^\alpha/\varepsilon^\beta\right)
	\quad\text{for some}\quad \alpha>1,\beta>0. $$
 Then, convergence \eqref{asa_nondeg} holds whenever $\alpha<\frac{1}{p}+1$ and $\beta+1>\alpha$. 
\end{remark}

Of course, we may state an integrated in time version of this Lemma with a weaker convergence in the conclusion. Since we integrate the pointwise estimates, the proof is the same.

\begin{lemma}\label{lem11-int} Let us suppose that there exist $0\le t_0< t_1<\infty$ such that
$$
\sup_{\eps>0} \int_{t_0}^{t_1}E_\eps (u^\eps(s)) \dd s <\infty.
$$
(a) If  we take $p= 1$ in (\ref{asa_nondeg}),
then
$$
 \lim_{\varepsilon\to 0}(\gamma^\varepsilon(u^\varepsilon) - m_\varepsilon(u^\eps)) =0 
 \qquad\hbox{ in }L^2(\partial\Omega\times(t_0, t_1)).
$$
(b) If we take $p= 1/2$ in (\ref{asa_nondeg}),
then
$$
 \lim_{\varepsilon\to 0}(\gamma^\varepsilon(u^\varepsilon) - m_\varepsilon(u^\eps)) =0\qquad\hbox{ in }L^1(\partial\Omega\times(t_0, t_1)).
$$ 
\end{lemma}

\section{Weak solutions}
The average operator $m_\eps$ defined in (\ref{def-we}) will play the major role in this section. Moreover, we will see the significance of the behavior of $\phi$ for small $\eps$. Namely, we assume that the following limit exists, although it may be infinite:
\begin{equation}\label{df-KA}
\lim_{\eps\to 0} \phi(\eps) \eps = \kappa \in [0, \infty] .
\end{equation}

\subsection{Abstract Cauchy problem} 
In the setting of linear equations that we consider here, there are well-known existence results for weak solutions based on the Galerkin method or on a parabolic version of the Lax-Milgram's lemma due to Lions. Let us recall the statements from \cite[Section III.2]{showalter} in an abbreviated version that is sufficient for our needs. 

Let $H$ be a Hilbert space identified with its dual and let $V$ be a separable Hilbert space with dual $V'$. We will denote by $(\cdot,\cdot)$ the inner product in $H$ and by $\langle\cdot,\cdot\rangle$ the pairing between $V'$ and $V$. Assume that $V$ is continuously and densely embedded in $H$. The transpose of this embedding gives an embedding of $H$ into $V'$. Let $A \colon V \to V'$ be a bounded linear operator. If we write  $\mathcal V := L^2(0,T;V)$, then $\mathcal V$ is a Hilbert space with dual $\mathcal V' = L^2(0,T; V')$. We recall \cite[Proposition III.1.2.]{showalter} that $\mathcal W : = \{v \in \mathcal V \colon v_t \in \mathcal V'\}$ is continuously embedded in $C([0,T], H)$. In fact, any $v \in \mathcal W$ is absolutely continuous and 
\[\frac{1}{2} \frac{\dd}{\dd t} \|v\|^2_H = \langle v_t, v \rangle \quad \text{for a.\,e.\ } t \in (0,T).\]

Given $u_0 \in H$, we consider the problem of finding such $u \in \mathcal V$ that 
\begin{equation} \label{weak_eqn} 
u_t + Au = 0, \quad u(0) = u_0. 
\end{equation}
Note that implicit in \eqref{weak_eqn} is condition $u_t \in \mathcal V'$. However, \eqref{weak_eqn} is equivalent to a seemingly weaker formulation, see \cite[Proposition III.2.1]{showalter},
\begin{multline} \label{weak_eqn_int} 
- \int_0^T \left(u(t), \varphi_t(t)\right) \dd t + \int_0^T \left\langle A u(t), \varphi(t)\right\rangle \dd t = \left(u_0, \varphi(0)\right) \\
\text{for every } \varphi \in \Phi:=\{v \in \mathcal V \colon v_t \in L^2(0,T;H), 
\ \varphi(T)=0\} . 
\end{multline}
In particular, any solution $u$ to \eqref{weak_eqn_int} belongs to $\mathcal W$ and satisfies the energy equality
\begin{equation} \label{energy_eq}
\frac{1}{2} \|u(t)\|_H^2 + \int_0^t \langle A u(s), u(s)\rangle \dd s = \frac{1}{2} \|u_0\|_H^2 \quad \text{for } t \in [0,T]. 
\end{equation} 
We have, see \cite[Proposition III.2.3]{showalter},
\begin{proposition}\label{weak_ex} 
Suppose that there exists $c > 0$ such that 
\begin{equation} \label{coercivity} 
\langle A v, v \rangle + \|v\|_H^2 \geq c \|v\|_V^2 \quad \text{for } v \in V.
\end{equation} 
Then, for any $u_0 \in H$ there exists a unique solution to \eqref{weak_eqn} (equivalently \eqref{weak_eqn_int}). 
\end{proposition}
We will apply this proposition to problems we are interested in.

\subsection{The approximate problem}

The approximate problems with a pronounced boundary layer are of the form
\begin{equation}\label{s2-strong}
\begin{array}{ll}
b_\varepsilon u^\eps_t = \di( a_\eps\nabla u^\eps ) & \text{in}\quad (0,T) \times \Omega, \\
a_\eps   \partial u^\eps/\partial\nu = 0  & \text{on}\quad (0,T) \times \partial\Omega, \\
u^\eps(x,0) =u_0(x) & \text{for}\quad x\in \Omega. 
\end{array}
\end{equation}
We take $b_\eps$ defined in (\ref{def-be}).

In this section we do not discuss the meaning of (\ref{s2-strong}$_2$) for weak solutions. However, for the strong solutions the situation is different, see Lemma \ref{lem1.1}. In fact the role of the boundary conditions here is to make the boundary integrals of the weak form vanish. 

System \eqref{s2-strong} can be set as an abstract Cauchy problem with choices 
\[H_\eps = L^2_\eps(\Omega),
\]
where $L^2_\eps(\Omega)$ was introduced in (\ref{df-Le}) and
\[V_\eps = \left\{u \in L^2(\Omega) \colon \nabla u \in L^2_{loc}(\Omega), \int_{\Omega} a_\eps |\nabla u|^2 < \infty\right\}, \quad \langle A_\eps u^1, u^2 \rangle = \int_\Omega a_\eps \nabla u^1 \cdot \nabla u^2.\]
Here and hereafter we often suppress $\dd x$, $\dd y$, $\dd t$, $\dd s$ at the end of integral if there is no risk of confusion. Our choices of $H_\eps$, $V_\eps$, $A_\eps$ correspond to: 
\begin{definition} \label{weak_bl_def}
We say that $u \in L^2(0,T;V_\eps)$ is a weak solution to \eqref{s2-strong} with initial datum $u_0 \in H_\eps$ if 
\begin{equation} \label{s2-weak}
 \int_\Omega b_\eps u_0 \varphi(0) + \int_0^T\!\! \int_\Omega b_\eps u \varphi_t  + \int_0^T \!\! \int_\Omega a_\eps \nabla u \cdot \nabla \varphi = 0 \quad \text{for every } \varphi \in \Phi_\eps, 
\end{equation} 
where 
\[\Phi_\eps = \{ \varphi \in L^2(0,T; V_\eps) \colon \varphi_t \in L^2(0,T; H_\eps),\ \varphi(T)=0\}.\]
\end{definition}
Proposition \ref{weak_ex} yields: 
\begin{proposition} \label{weak_bl_ex} 
 For any $u_0 \in H_\eps$ there exists exactly one weak solution to \eqref{s2-strong}. 
\end{proposition} 

\subsection{The dynamic boundary problem} 
We aim to prove that weak solutions to \eqref{s2-strong} converge to (yet undefined) weak solutions to \eqref{eq1}, provided that $\kappa \in (0, \infty)$. One way to define them is to consider regular enough solutions to \eqref{eq1} and test them with a function $\varphi \in C^1([0,T]\times \overline{\Omega})$ such that $\varphi(T,\cdot) = 0$. Multiplying equation $u_t = \Delta u$ by $\varphi$, integrating over $\Omega_T$ and integrating by parts yields 
\begin{equation*}
    - \int_\Omega u_0 \varphi(0,\cdot) - \int_0^T \!\! \int_\Omega u \varphi_t = \int_0^T \!\! \int_\Omega u_t \varphi = \int_0^T \!\! \int_\Omega \Delta u \varphi = - \int_0^T \!\! \int_\Omega \nabla u \cdot \nabla \varphi + \int_0^T \!\! \int_{\partial \Omega} \frac{ \partial u}{\partial \nu} \varphi.      
\end{equation*} 
Taking into account that $\kappa u_t + \frac{ \partial u}{\partial \nu} = 0$ on $\partial \Omega$ leads us to 
\begin{equation*}
    \int_0^T \!\! \int_{\partial \Omega} \frac{ \partial u}{\partial \nu} \varphi = - \kappa \int_0^T \!\! \int_{\partial \Omega} u_t \varphi =  \kappa \int_{\partial \Omega} u_0 \varphi(0,\cdot) + \kappa \int_0^T \!\! \int_{\partial \Omega} u \varphi_t.
\end{equation*} 
Summing up, we obtain the following equality constituting a naive notion of weak solution.
\begin{equation} \label{weak_naive}
\int_\Omega u_0 \varphi(0,\cdot) + \kappa \int_{\partial \Omega}\gamma  u_0 \varphi(0,\cdot) + \int_0^T \!\! \int_\Omega u \varphi_t + \kappa \int_0^T \!\! \int_{\partial \Omega} \gamma u \varphi_t + \int_0^T \!\! \int_\Omega \nabla u \cdot \nabla \varphi = 0.
\end{equation} 
However, it is not immediately clear how to obtain well-posedness for \eqref{weak_naive}. 

On the other hand, system \eqref{eq1} can be expressed 
as an abstract Cauchy problem with the choices 
\begin{equation}\label{df-Xka}
H=H_0^\kappa\!: = (L^2(\Omega) \times  L^2(\partial\Omega), (\cdot, \cdot)_\kappa),\qquad
((u_1, \xi_1), (u_2,\xi_2))_\kappa \! :=
\int_{\Omega} u_1 u_2\, dx + 
\kappa 
\int_{\partial\Omega} \xi_1 \xi_2 \, d\cH^{N-1}.
\end{equation}
\[V=V_0\!: = \left\{(u,w) \in H^1(\Omega)\times L^2(\partial \Omega) \colon \gamma u = w\right\}, \quad \langle A(u, w) , (\varphi, \psi) \rangle = A_0 (u, w) , (\varphi, \psi) \rangle \!:= \int_\Omega \nabla u \cdot \nabla \varphi.\]
Having these definitions, we can introduce: 
\begin{definition} \label{weak_dbc_def}
We say that $(u,w) \in L^2(0,T;V_0)$ is a weak solution to the dynamic boundary value problem (b.v.p.) for the heat equation with initial datum $(u_0, w_0) \in H_0^\kappa$ if 
\begin{multline} \label{weak_dbc}
 \int_\Omega u_0 \varphi(0) + \kappa \int_{\partial \Omega} w_0 \psi(0) + \int_0^T\!\! \int_\Omega u \varphi_t + \kappa \int_0^T\!\! \int_{\partial \Omega} w \psi_t + \int_0^T \!\! \int_\Omega \nabla u \cdot \nabla \varphi = 0 \\ \text{for every } (\varphi, \psi) \in \Phi_0, 
\end{multline} 
where 
\[ \Phi_0 = \{ (\varphi, \psi) \in L^2(0,T; V_0) \colon (\varphi_t, \psi_t) \in L^2(0,T; H_0), \ (\varphi, \psi)(T)=0\}.\]
\end{definition}
Again, by Proposition \ref{weak_ex}, 
\begin{proposition} \label{weak_dbc_ex} 
 For any $(u_0, w_0) \in H_0^\kappa$ there exists exactly one weak solution to the dynamic b.v.p.\ for the heat equation. 
\end{proposition} 
We note that if we restrict our attention to the case when $\varphi \in C^1([0,T]\times \overline\Omega)$ and $w_0 = \gamma u_0$, then \eqref{weak_dbc} becomes \eqref{weak_naive}. However, Proposition \ref{weak_dbc_ex} indicates that, even though $w(t) = \gamma u(t)$ for a.\,e.\ $t \in ]0,T[$ for any solution $(u,w)$, initial data $(u_0, w_0)$ can be chosen independently and different choices of $w_0$ lead to different solutions. Taking this into account, we propose the following relaxation of the definition of weak solution that will be useful in the sequel. 
\begin{proposition} \label{prop_equiv}
Let $(u_0,w_0) \in L^2(\Omega)\times L^2(\partial \Omega)$. Suppose that $u \in L^2(0,T, H^1(\Omega))$. Then $(u, \gamma u)$ is a weak solution to the dynamic b.v.p.\ for the heat equation with initial datum $(u_0, w_0)$ if and only if 
\begin{multline} \label{relaxed_dbc}
 \int_\Omega u_0 \varphi(0) + \kappa \int_{\partial \Omega} w_0 \varphi(0) + \int_0^T\!\! \int_\Omega u \varphi_t + \kappa \int_0^T\!\! \int_{\partial \Omega} \gamma u \varphi_t + \int_0^T \!\! \int_\Omega \nabla u \cdot \nabla \varphi = 0 \\ \text{for every } \varphi \in C^1([0,T]\times \overline{\Omega}) \text{ such that } \varphi(T)=0. 
\end{multline}
\end{proposition}
The absence of the second component $w$ in the rephrased definition is just a semantic difference. More importantly, we restrict the class of test functions. It is easy to see that if $(u, \gamma u)$ is a weak solution, then $u$ satisfies \eqref{relaxed_dbc}. The reverse implication is a straightforward consequence of the following approximation lemma.   
\begin{lemma} \label{weak_approx} 
Let $(\varphi, \psi) \in H^1(0,T; L^2(\Omega)\times L^2(\partial \Omega))$ be such that $\varphi \in L^2(0,T;H^1(\Omega))$, $\gamma \varphi = \psi$. Then there exists a sequence $\varphi^k \in C^\infty([0,T]\times \overline{\Omega})$ such that $(\varphi^k, \gamma \varphi^k) \to (\varphi, \psi)$ in $\Phi_0$, i.\,e.
\[ (\varphi^k, \gamma \varphi^k) \to (\varphi, \psi) \text{ in } H^1(0,T; L^2(\Omega)\times L^2(\partial \Omega)), \qquad \varphi^k \to \varphi \text{ in } 
L^2(0,T;H^1(\Omega)).\]
Furthermore, if $(\varphi,\psi)(T)=0$, then $\varphi^k(T) = 0$.
\end{lemma} 
\begin{proof} 
Let us pick any $(\varphi, \psi) \in \Phi_0$. We denote by $\overline \varphi, \overline \psi$ its extensions outside $[0,T]$. Namely, if $(\varphi,\psi)(T)=0$, then we extend the couple $(\varphi,\psi))$ past $T$ by odd reflection, otherwise by even reflection. We extend it for times earlier than $0$ by (say) even reflection. Let $\overline{\varphi}_h$ be given by 
\[\overline{\varphi}_h(t, x', x_N) = \left\{\begin{array}{l} 
\overline{\psi}(t, x', 0) \text{ if } x_N < h, \\ 
\overline{\psi}(t, x', 1) \text{ if } x_N > 1-h, \\ 
\overline{\varphi}\left(t, x', \frac{x_N-h}{1-2h}\right) \text{ otherwise. }
\end{array} \right. \]
Then, we define $\varphi^{\delta', h} \in L^2(0,T;H^1(\Omega))$ and $\varphi^{\delta, \delta', h} \in C^\infty([0,T]\times \overline{\Omega})$ as restrictions to $[0,T]\times\overline{\Omega}$ of 
\[\overline{\varphi}^{\delta', h} = \varrho^{x'}_{\delta'} * (\overline{\varphi}_h), \qquad \overline{\varphi}^{\delta, \delta', h} = \varrho^t_\delta * (\varrho^{x_N}_\delta * (\varrho^{x'}_{\delta'} * (\overline{\varphi}_h))).\]
Here $\varrho_{\delta'}^{x'}$ denotes the standard mollifier in $x'$ direction while $\varrho_\delta^t$, $\varrho_\delta^{x_N}$ denote the mollifiers in $t$ and $x_N$, respectively.

Let us check that $\varphi^{\delta', h} \in \Phi_0$. The non-trivial part of this task is  establishing the weak differentiability in the direction of $x_N$. Let $\vartheta^\eta = \vartheta^\eta(x_N)$ be the continuous function on $\bR$ such that: 
\begin{itemize} 
\item $\vartheta^\eta = 1$ on $[h-\eta,h+\eta]\cup [1-h-\eta, 1 - h + \eta]$, 
\item $\vartheta^\eta = 0$ outside $[h- 2\eta, h+ 2\eta] \cup [1-h- 2\eta, 1-h+ 2\eta]$,
\item $\vartheta^\eta$ coincides with affine functions on $[h- 2\eta,h-\eta]$, $[h+\eta, h+2\eta]$, $[1-h- 2\eta,1-h-\eta]$, and $[1-h+\eta,1-h+2\eta]$. 
\end{itemize} 
Then, for any $w \in C_c^\infty(]0,T[ \times \Omega)$, 
\begin{multline} 
\int_0^T \!\!\! \int_\Omega \varphi^{\delta', h} w_{x_{N}} = \int_0^T \!\!\! \int_\Omega \varphi^{\delta', h} (\vartheta^\eta w)_{x_{N}} + \int_0^T \!\!\! \int_\Omega \varphi^{\delta', h} ((1-\vartheta^\eta) w)_{x_{N}} \\ 
= \int_0^T \!\!\! \int_{\Omega} \varphi^{\delta', h} \vartheta^\eta_{x_{N}} w + \int_0^T \!\!\! \int_\Omega \varphi^{\delta', h} \vartheta^\eta w_{x_{N}} - \int_0^T \!\!\! \int_\Omega \varphi^{\delta', h}_{x_{N}} (1-\vartheta^\eta) w = I_1^\eta + I_2^\eta + I_3^\eta.  
\end{multline}
Using a change of variables, 
\begin{multline*} 
I_1^\eta = \frac{1}{\eta}\int_0^T \!\!\! \int_{\bT^{n-1}} \!\! \int_{h-2\eta}^{h-\eta} \varrho^{x'}_{\delta'} * \psi(\cdot,\cdot,0)\, w - \frac{1-2h}{\eta}\int_0^T \!\!\! \int_{\bT^{n-1}} \!\! \int_{\frac{\eta}{1-2h}}^{\frac{2\eta}{1-2h}} \varrho^{x'}_{\delta'} * \varphi\, w(\cdot,\cdot, h + (1-2h)(\cdot)) \\ + \frac{1-2h}{\eta}\int_0^T \!\!\! \int_{\bT^{n-1}} \!\! \int_{1-\frac{2\eta}{1-2h}}^{1-\frac{\eta}{1-2h}} \varrho^{x'}_{\delta'} * \varphi\, w(\cdot,\cdot, h + (1-2h)(\cdot)) - \frac{1}{\eta}\int_0^T \!\!\! \int_{\bT^{n-1}} \!\! \int_{1- h+\eta}^{1-h+2\eta} \varrho^{x'}_{\delta'} * \psi(\cdot,\cdot,1)\, w.
\end{multline*} 
By a property of trace, 
\begin{multline*} 
\lim_{\eta \to 0^+} I_1^\eta = \int_0^T \!\!\! \int_{\bT^{n-1}} \varrho^{x'}_{\delta'} * \psi(\cdot,\cdot,0)\, w(\cdot,\cdot,h) - \int_0^T \!\!\! \int_{\bT^{n-1}} \gamma(\varrho^{x'}_{\delta'} * \varphi)(\cdot,\cdot,0)\, w(\cdot,\cdot, h) \\ + \int_0^T \!\!\! \int_{\bT^{n-1}} \gamma(\varrho^{x'}_{\delta'} * \varphi)(\cdot,\cdot,1)\, w(\cdot,\cdot, 1-h) - \int_0^T \!\!\! \int_{\bT^{n-1}} \varrho^{x'}_{\delta'} * \psi(\cdot,\cdot,1)\, w(\cdot,\cdot,1-h) = 0 ,
\end{multline*}
because we used here $\gamma(\varrho^{x'}_{\delta'} * \varphi) = \varrho^{x'}_{\delta'} * \gamma\varphi$. It is not difficult to check that $I_2^\eta \to 0$ and 
\[I_3^\eta =  - \int_0^T \!\!\! \int_{\Omega_h} \frac{1}{1-2h}\varrho^{x'}_{\delta'}*\varphi_x\left(\cdot,\cdot,\tfrac{(\cdot)-h}{1-2h}\right) (1-\vartheta^\eta) w \to - \int_0^T \!\!\! \int_{\Omega_h} \frac{1}{1-2h}\varrho^{x'}_{\delta'}*\varphi_x\left(\cdot,\cdot,\tfrac{(\cdot)-h}{1-2h}\right) w\]
as $\eta \to 0^+$. Thus, we see that $\varphi^{\delta', h} \in L^2(0,T;H^1(\Omega))$ and 
\[ \varphi^{\delta', h}_x = \frac{1}{1-2h} \mathbf{\chi}_{\Omega_h} \varrho^{x'}_{\delta'}*\varphi_x\left(\cdot,\cdot,\tfrac{(\cdot)-h}{1-2h}\right).\]
It is now easy to check that $(\varphi^{\delta,\delta',h}, \gamma \varphi^{\delta,\delta',h})  \to (\varphi^{\delta',h}, \gamma \varphi^{\delta',h})$ in $\Phi_0$ as $\delta \to 0^+$, $(\varphi^{\delta',h}, \gamma \varphi^{\delta',h}) \to (\varphi^{\delta'}, \gamma \varphi^{\delta'})$ in $\Phi_0$ as $h \to 0^+$ and $(\varphi^{\delta'}, \gamma \varphi^{\delta'})  \to (\varphi, \psi)$ in $\Phi_0$ as $\delta \to 0^+$. With a diagonal procedure, we can extract sequences $\delta_k$, $h_k$, $\delta'_k$ such that for $\varphi^k = \varphi^{\delta_k,\delta'_k,h_k}$ we have $(\varphi^k, \gamma \varphi^k) \to (\varphi, \psi)$ in $\Phi_0$ as $k \to \infty$.  
\end{proof} 

\subsection{Neumann and Dirichlet problems}

For completness, we also recall how (Cauchy-)Neumann and (Cauchy-)Dirichlet problems for the heat equation fit into the abstract framework. 

The classical Neumann problem is to find, for a given $u_0$, a solution $u$ to
\begin{equation}\label{eqN}
\left\{
\begin{array}{ll}
u_t = \Delta u &  \hbox{in } \Omega_T, \\
 \frac{\partial u }{\partial \nu} =0 & \hbox{in } \partial\Omega\times (0,T), \\
u(0,x) = u_0(x) & \hbox{in }\Omega,
\end{array}
\right.
\end{equation}
The system \eqref{eqN} can be expressed as an abstract Cauchy problem with the choices 
\[H = L^2(\Omega),\quad  V = H^1(\Omega), \quad \langle A(u), \varphi\rangle = \int_\Omega \nabla u \cdot \nabla \varphi, \] 
corresponding to:
\begin{definition} \label{weak_N_def}
We say that $u \in L^2(0,T;H^1(\Omega))$ is a weak solution to the Neumann problem for the heat equation with initial datum $u_0 \in L^2(\Omega)$ if 
\begin{equation} \label{eqN-weak}
 \int_\Omega u_0 \varphi(0) + \int_0^T\!\! \int_\Omega u \varphi_t  + \int_0^T \!\! \int_\Omega \nabla u \cdot \nabla \varphi = 0 \quad \text{for every } \varphi \in \Phi_0^0, 
\end{equation} 
where 
\[\Phi_0^0 = \{ \varphi \in L^2(0,T; H^1(\Omega)) \colon \varphi_t \in L^2(0,T; L^2(\Omega)),\ \varphi(T)=0\}.\]
\end{definition}
Proposition \ref{weak_ex} yields: 
\begin{proposition} \label{weak_N_ex} 
 For any $u_0 \in L^2(\Omega)$ there exists exactly one weak solution to the Neumann problem for the heat equation with initial datum $u_0$. 
\end{proposition} 

Similarly, the classical Cauchy-Dirichlet problem is to find, for a given $u_0$, a solution $u$ to
\begin{equation}\label{eqD}
\left\{
\begin{array}{ll}
u_t = \Delta u &  \hbox{in } \Omega_T, \\
 u =0 & \hbox{in } \partial\Omega\times (0,T), \\
u(0,x) = u_0(x) & \hbox{in }\Omega,
\end{array}
\right.
\end{equation}
The system \eqref{eqD} can be expressed as an abstract Cauchy problem with the choices 
\[H = L^2(\Omega),\quad  V = H^1_0(\Omega), \quad \langle A(u), \varphi\rangle = \int_\Omega \nabla u \cdot \nabla \varphi. \] 
This corresponds to:
\begin{definition} \label{weak_D_def}
We say that $u \in L^2(0,T;H^1_0(\Omega))$ is a weak solution to the Dirichlet problem for the heat equation with initial datum $u_0 \in L^2(\Omega)$ if 
\begin{equation} \label{eqD-weak}
 \int_\Omega u_0 \varphi(0) + \int_0^T\!\! \int_\Omega u \varphi_t  + \int_0^T \!\! \int_\Omega \nabla u \cdot \nabla \varphi = 0 \quad \text{for every } \varphi \in \Phi_0^\infty, 
\end{equation} 
where 
\[\Phi_0^\infty = \{ \varphi \in L^2(0,T; H^1_0(\Omega)) \colon \varphi_t \in L^2(0,T; L^2(\Omega)),\ \varphi(T)=0\}.\]
\end{definition}
Proposition \ref{weak_ex} yields: 
\begin{proposition} \label{weak_D_ex} 
 For any $u_0 \in L^2(\Omega)$ there exists exactly one weak solution to the Dirichlet problem for the heat equation with initial datum $u_0$. 
\end{proposition} 

\subsection{Convergence to the dynamic boundary value problem} \label{subs-co}

In this subsection, we consider the case $\kappa\in(0,\infty)$. We say that $\{u_0^\eps\}_{\eps>0}$ is a well prepared sequence of approximating initial data for $(u_0, w_0)$ if 
\begin{equation}\label{init_conv} 
\sup_{\eps > 0} \int_{\Omega} b_{\eps} (u_0^{\eps})^2 =:M< \infty \quad \text{and} \quad (u_0^\eps, m_\eps u_0^\eps) \to (u_0, w_0) \quad \text{in } L^2(\Omega) \times L^2(\partial \Omega).
\end{equation} 
Note that for any pair $(u_0,w_0) \in L^2(\Omega) \times L^2(\partial \Omega)$ there exists such a sequence. Indeed, $u_0^\eps = \chi_{\Omega_\eps} u_0 + \chi_{\Omega \setminus \Omega_\eps} w_0 $ satisfies \eqref{init_conv}. 

\begin{theorem}\label{thm_conv_dbc}
Let $(u_0,w_0) \in L^2(\Omega) \times L^2(\partial \Omega)$ and let $(u,w)$ be the corresponding  weak solution to the dynamic b.v.p.\ for the heat equation. We assume that $\kappa\in(0,\infty)$ and the non-degeneracy condition \eqref{asa_nondeg} holds with $p=\frac12$. 

Suppose $\{u_0^\eps\}_{\eps>0}$ is a well prepared sequence of approximating initial data for $(u_0, w_0)$ and $\{u^\eps\}_{\eps>0}$ are the corresponding weak solutions to \eqref{s2-strong}. Then,  
\begin{equation} \label{path_conv} 
(u^\eps, m_\eps u^\eps) \rightharpoonup (u, w) \quad \text{in } L^2(0,T;L^2(\Omega)\times L^2(\partial \Omega)).
\end{equation} 
\end{theorem}
\begin{remark} 
If $u_0 \in H^1(\Omega)$, then $\gamma u_0 \in L^2(\partial \Omega)$ is well defined and $u_0^\eps \equiv u_0$ satisfies \eqref{init_conv} with $w_0 = \gamma u_0$. In this way we recover a conceptually simpler result, formally similar to \cite[Theorem 3.1]{colli}, where only the homogeneous initial condition $u_0 = 0$ was considered. 
\end{remark} 
\begin{proof} 
Let us take any $u_0^\eps$ satisfying \eqref{init_conv} and let $u^\eps$ be the weak solution to \eqref{s2-strong} with initial datum $u_0^\eps$. The energy equality \eqref{energy_eq} for \eqref{s2-strong} takes the form 
\begin{equation}\label{energy_eq_bl} 
\frac{1}{2} \int_\Omega b_\eps (u^\eps(t))^2 + \int_0^t \!\! \int_\Omega a_\eps |\nabla u^\eps|^2 = \frac{1}{2} \int_{\Omega} b_{\eps} (u_0^{\eps})^2. 
\end{equation} 
Since $\kappa\in (0,\infty)$, 
\begin{equation} \label{gamma_b_est}
\frac\kappa2\int_{\partial \Omega} (m_\eps u^\eps(t))^2 + \int_\Omega (u^\eps(t))^2 \le \int_\Omega b_\eps (u^\eps(t))^2 
\end{equation}
for $t \in [0,T]$ and sufficiently small $\eps>0$. 
Thus, by \eqref{init_conv}, \eqref{energy_eq_bl} and \eqref{gamma_b_est}, $u^\eps$ is bounded in $L^2(0,T;L^2(\Omega))$ and $m_\eps u^\eps$ is bounded in $L^2(0,T;L^2(\partial \Omega))$. Hence, we can extract weakly convergent sequences $u^{\eps_k} \rightharpoonup \overline{u}$ in $L^2(0,T;L^2(\Omega))$ and $m_{\eps_k} u^{\eps_k} \rightharpoonup \overline{w}$ in $L^2(0,T;L^2(\partial \Omega))$. 

We need to prove that the pair $(\overline{u},\overline{w})$ is the weak solution to the dynamic b.v.p. Let us take any $\varphi \in C^1([0,T]\times \overline{\Omega})$ such that $\varphi(T)=0$. Surely, $\varphi$ is a legitimate test function for the weak formulation of the boundary layer problem \eqref{s2-weak}. By \eqref{init_conv},  
\begin{equation} \label{weak_pass_init} 
\int_\Omega b_\eps u_0^\eps \varphi(0, \cdot) 
= \eps\phi(\eps) m_\eps(u^\eps_0 \vfi(0,\cdot)) + \int_{\Omega_\eps} u_0^\eps \varphi(0, \cdot) 
\to \kappa\int_{\partial \Omega} w_0  \varphi(0,\cdot) + \int_\Omega u_0 \varphi(0,\cdot).
\end{equation} 
In this limit we also use the uniform continuity of $\vfi$ over $\overline\Omega.$

Let us fix $\delta>0$, we will consider $\eps\in(0,\delta).$ Then, we have 
\[\int^T_0 \!\!\! \int_\Omega b_{\eps} u^{\eps} \varphi_t  = \int^T_0 \!\!\! \int_{\Omega_\eps} u^{\eps} \varphi_t + \eps \phi(\eps)\int^T_0 \!\!\! \int_{\partial \Omega} m_{\eps} u^{\eps} \varphi_t(\cdot,\cdot,0) + \phi(\eps) \int^T_0 \!\!\! \int_{\Omega\setminus \Omega_{\eps}}  u^{\eps} \left(\varphi_t - \varphi_t(\cdot,\cdot,0)\right),\]
where, since $\eps \phi(\eps) \to \kappa \in (0, \infty)$, 
\begin{multline*} 
\left|\phi(\eps) \int^T_0 \!\!\! \int_{\Omega\setminus \Omega_{\eps}}  u^{\eps} \left(\varphi_t - \varphi_t(\cdot,\cdot,0)\right)\right| \\ \leq \left(\phi(\eps) \int^T_0 \!\!\! \int_{\Omega\setminus \Omega_{\eps}} (u^{\eps})^2 \right)^\frac{1}{2} \left(\phi(\eps) \int^T_0 \!\!\! \int_{\Omega\setminus \Omega_{\eps}} \left(\varphi_t - \varphi_t(\cdot,\cdot,0)\right)^2 \right)^\frac{1}{2} \to 0
\end{multline*} 
by \eqref{energy_eq_bl} and smoothness of $\varphi$. Therefore, we can pass to the limit 
\begin{equation} \label{weak_pass_t} 
\int^T_0 \!\!\! \int_\Omega b_{\eps_k} u^{\eps_k} \varphi_t  \to \int^T_0 \!\!\! \int_\Omega \overline{u} \varphi_t + \kappa\int^T_0 \!\!\! \int_{\partial \Omega} \overline{w} \varphi_t.
\end{equation} 
Furthermore, again by \eqref{energy_eq_bl}, \eqref{init_conv}, we can assume without the loss of generality that $\nabla u_{\eps_k} \rightharpoonup \nabla \overline{u}$ in $L^2(0,T;L^2_{loc}(\Omega))$ and 
\[\left|\int^T_0 \!\!\! \int_\Omega a_{\eps} \nabla u^{\eps} \nabla \varphi - \int^T_0 \!\!\! \int_{\Omega_\delta} a_{\eps} \nabla u^{\eps} \nabla \varphi\right|^2 = \left|\int^T_0 \!\!\! \int_{\Omega\setminus \Omega_\delta} a_{\eps} \nabla u^{\eps} \nabla \varphi\right|^2 \leq 2M 
\int^T_0 \!\!\! \int_{\Omega \setminus \Omega_\delta} |\nabla \varphi|^2.\]  
For a fixed $\delta>0$, the uniform convergence of $a_{\eps}$ to 1 on $\Omega_\delta$ and the weak convergence of $\nabla u^{\eps_k}$ in this region yield
\[\lim
_{k \to \infty} 
\int^T_0 \!\!\! \int_{\Omega_\delta} a_{\eps_k} \nabla u^{\eps_k} \nabla \varphi = 
\int^T_0 \!\!\! \int_{\Omega_\delta} \nabla \overline{u} \nabla \varphi.
\]
At the same time
$$
\left|\int^T_0 \!\!\! \int_{\Omega\setminus\Omega_\delta} \nabla \overline{u} \nabla \varphi \right| \le \| \nabla \overline{u}\|_{L^2((\Omega\setminus\Omega_\delta)\times (0,T))} \| \nabla \varphi\|_{L^2((\Omega\setminus\Omega_\delta)\times (0,T))}.
$$
Thus, passing to the limit $\delta \to 0^+$, we obtain
\begin{equation} \label{weak_pass_nabla}
\lim_{k \to \infty} \int^T_0 \!\!\! \int_\Omega a_{\eps_k} \nabla u^{\eps_k} \nabla \varphi = \int^T_0 \!\!\! \int_\Omega \nabla \overline{u} \nabla \varphi.
\end{equation} 
Summing up \eqref{weak_pass_init}, \eqref{weak_pass_t} and \eqref{weak_pass_nabla} we obtain 
\begin{equation} \label{almost_weak_dbc} 
\int_\Omega u_0 \varphi(0,\cdot) + \kappa \int_{\partial \Omega} w_0  \varphi(0,\cdot) + \int^T_0 \!\!\! \int_\Omega \overline{u} \varphi_t + \kappa \int^T_0 \!\!\! \int_{\partial \Omega} \overline{w} \varphi_t + \int^T_0 \!\!\! \int_\Omega \nabla \overline{u} \nabla \varphi = 0.
\end{equation} 
It remains to show that $\overline{w} = \gamma \overline{u}$. Indeed, this follows after combining (\ref{energy_eq_bl}) with Lemma \ref{traceU-w} and Lemma \ref{lem11-int}. Thus, by Proposition \ref{prop_equiv}, $(\overline{u},\overline{w})$ coincides with the weak solution $(u,w)$. Since we can extract a convergent subsequence from any subsequence of $(u^\eps, m_\eps u^\eps)$, the whole sequence converges to $(u, w)$ due to uniqueness of the weak solution to the limiting problem. 
\end{proof} 

\subsection{Convergence in the cases $\kappa =0$ and $\kappa = +\infty$} 
\label{k0infty}
We say that $\{u_0^\eps\}_{\eps>0} \subset L^2(\Omega)$ is a well prepared sequence of approximating initial data for $u_0$ if 
\begin{equation} \label{init_conv_ND} 
\sup_\eps \int_\Omega b_\eps (u_0^\eps)^2 <\infty \quad \text{and} \quad u_0^\eps \rightharpoonup u_0 \qquad\hbox{in }L^2(\Omega).
\end{equation} 
Note that for any $u_0 \in L^2(\Omega)$ there exists such a sequence. Indeed, $u_0^\eps = \chi_{\Omega_\eps} u_0 $ satisfies \eqref{init_conv}, no matter the choice of $\kappa$. 

The case $\kappa =0$ is relatively easy. Namely, we can see that the boundary terms vanish in the limit. This corresponds to the case of homogeneous Neumann data for the limit function $u$.

\begin{proposition}\label{w-ka-zer}
Suppose that $\kappa = 0$. Let $u_0 \in L^2(\Omega)$ and let $u$ be the corresponding weak solution to the Neumann problem. Suppose that a family $u_0^\eps$ satisfies \eqref{init_conv_ND} and $u^\eps$ is the corresponding family of solutions to (\ref{s2-strong}). Then 
$$
u^\eps \rightharpoonup u\quad\text{in} \quad L^2(0,T; L^2(\Omega)) \cap L^2(0,T; H^1_{loc}(\Omega)). 
$$
\end{proposition}
\begin{proof}
Using \eqref{energy_eq_bl}, we extract a subsequence $u^{\eps_k}$ such that  
\[u^{\eps_k} \rightharpoonup \overline{u}\quad\text{in} \quad L^2(0,T; L^2(\Omega)) \cap L^2(0,T; H^1_{loc}(\Omega)). \] 
We need to prove that $\overline{u}$ is the solution to the Neumann problem. As in the proof of Theorem \ref{thm_conv_dbc}, we show that 
$$
\lim_{k \to \infty} \int^T_0 \!\!\! \int_\Omega a_{\eps_k} \nabla u^{\eps_k} \nabla \varphi = \int^T_0 \!\!\! \int_\Omega \nabla \overline{u} \nabla \varphi,
$$
see (\ref{weak_pass_nabla}).
It remains to prove that 
$$
\lim_{k \to \infty} \int_0^T\int_\Omega b_{\eps_k} u^{\eps_k} \vfi_t= \int_0^T\int_\Omega  \overline{u} \vfi_t. 
$$
In order to achieve this goal it is enough to notice that
$$
\lim_{\eps \to 0^+} \int_0^T\int_{\Omega\setminus\Omega_{\eps}} b_{\eps} u^{\eps} \vfi_t= 0.
$$
Indeed,
$$
\int_0^T\int_{\Omega\setminus\Omega_{\eps}} |b_{\eps} u^{\eps} \vfi_t| \le
\int_0^T\int_{\partial\Omega}\left( \int_0^{\eps}\phi(\eps)  (u^{\eps})^2\right)^{1/2}
\left(\int_0^{\eps} \phi({\eps}) (\vfi_t)^2\right)^{1/2}.
$$
Now, by (\ref{energy_eq_bl}), the first factor on the RHS is finite while
$$
\phi({\eps})\int_0^{\eps} (\vfi_t)^2 = {\eps} \phi({\eps}) m_{\eps}((\vfi_t)^2)
$$
converges to zero. This is so because $\kappa=0$ and the test function $\vfi$ has bounded derivatives. Finally, from the uniqueness of solutions to the Neumann problem, we deduce that the whole sequence $u^\eps$ converges to $u$ by the usual argument. 
\end{proof}
Let us stress  that in this proposition we do not use  assumption (\ref{asa_nondeg}).

\bigskip
Finally, we state the companion result for $\kappa=\infty$.
\begin{proposition}\label{w-ka-inft}
Suppose that $\kappa = +\infty$ and that $a_\eps$ satisfies \eqref{asa_nondeg} with $p=\frac{1}{2}$. Let $u_0 \in L^2(\Omega)$ and let $u$ be the corresponding weak solution to the Dirichlet problem. Suppose that a family $u_0^\eps$ satisfies \eqref{init_conv_ND} and $u^\eps$ is the corresponding family of solutions to (\ref{s2-strong}). Then 
$$
u^\eps \rightharpoonup u \text{ in } L^2(0,T; L^2(\Omega)) \cap L^2(0,T; H^1_{loc}(\Omega)) \quad \text{and} \quad m_\eps u^\eps \to 0 = \gamma u \text{ in } L^2(0,T; L^2(\partial \Omega)) 
$$
\end{proposition}
\begin{proof}
Let $u_0^\eps$ be any sequence satisfying \eqref{init_conv_ND}. Then, by \eqref{energy_eq_bl} and the uniform bound in \eqref{init_conv_ND}, there exists a subsequence $u^{\eps_k}$ and $\overline{u} \in L^\infty(0,T;L^2(\Omega))\cap L^2(0,T;H^1(\Omega))$ such that 
$$
u^\eps \rightharpoonup \overline{u} \text{ in } L^2(0,T; L^2(\Omega)) \cap L^2(0,T; H^1_{loc}(\Omega)). 
$$
Moreover, we have 
\begin{equation}\label{5.tr}
\| m_\eps(u^\eps) \|^2 = 
\int_0^T\int_{\partial\Omega} \left(\frac1\varepsilon\int_0^\varepsilon u^\eps\, dx_N \right)^2\,d x' 
\le 
\frac{\phi(\eps)}{\eps \phi(\eps)}\int_0^T \int_{\Omega\setminus \Omega_\eps} u^2_\eps\,dxdt \le \frac1{\eps \phi} \int_\Omega b_\eps u^2_{0,\eps}\,dx \stackrel{\eps \to 0}{\longrightarrow} 0.  
\end{equation}
Thus, by Lemmata \ref{traceU-w} and \ref{lem11-int}, 
\[\gamma \overline{u} = \lim_{ k \to \infty} \gamma^{\eps_k} u^{\eps_k} = \lim_{k \to \infty} m_{\eps_k} u^{\eps_k} = 0,\]
where the limits are understood in $L^2(0,T;L^2(\partial \Omega))$. Now take any $\varphi \in C^1_c([0,T[\times\Omega)$. Then, we easily check that \[\int_\Omega b_\eps u_0^\eps \varphi(0) \to \int_\Omega u_0 \varphi(0), \quad \int_0^T\!\!\int_\Omega b_\eps u^\eps \varphi_t \to \int_0^T\!\!\int_\Omega \overline{u} \varphi_t, \quad \int_0^T\!\!\int_\Omega a_\eps \nabla u^\eps \cdot \nabla \varphi \to \int_0^T\!\!\int_\Omega \nabla \overline{u} \cdot \nabla \varphi.\]   
Thus, $\overline{u}$ satisfies 
\[\int_\Omega u_0 \varphi(0) + \int_0^T\!\!\int_\Omega \overline{u} \varphi_t + \int_0^T\!\!\int_\Omega \nabla \overline{u} \cdot \nabla \varphi = 0\]
for any $\varphi \in C^1_c([0,T[\times\Omega)$. Since this set is dense in 
$$
\{\varphi \in L^2(0,T;H_0^1(\Omega))\cap H^1(0,T; L^2(\Omega)) \colon \varphi(T) = 0\},
$$ 
we conclude that $\overline u$ coincides with the weak solution to the Dirichlet problem $u$. Then, by the usual argument involving uniqueness of $u$, we deduce the weak convergence of the whole sequence $u^\eps$. 
\end{proof}

\subsection{Necessity of the non-degeneracy condition \eqref{asa_nondeg}}

In Theorem \ref{thm_conv_dbc} and Proposition \ref{w-ka-inft}, as opposed to Proposition \ref{w-ka-zer}, we require the non-degeneracy condition \eqref{asa_nondeg}. Here, we show that this assumption cannot be dropped. To see this, we can take any $a_\eps$ satisfying \eqref{asa} and 
\begin{equation} \label{asa_deg}
\max_{[\eps, 2 \eps]} a_\eps \leq C \eps^2
\end{equation}
with $C>0$, which contradicts the first part of assumption \eqref{asa_nondeg}.
\begin{proposition}\label{prop_deg}
Suppose that $\kappa \in [0, \infty]$ and $a_\eps$ satisfies \eqref{asa_deg}. Let $u_0 \in L^2(\Omega)$ and let $u$ be the corresponding weak solution to the Neumann problem. Suppose that a family $u_0^\eps$ satisfies \eqref{init_conv_ND} and $u^\eps$ is the corresponding family of solutions to (\ref{s2-strong}). Then 
$$
u^\eps \rightharpoonup u\quad\text{in} \quad L^2(0,T; L^2(\Omega)) \cap L^2(0,T; H^1_{loc}(\Omega)). 
$$
\end{proposition} 
\begin{proof} 
As in subsection \ref{k0infty}, we see that there exists a subsequence $u^{\eps_k}$ and $\overline{u} \in L^\infty(0,T;L^2(\Omega))\cap L^2(0,T;H^1(\Omega))$ such that 
\begin{equation} \label{deg_conv}
u^{\eps_k} \rightharpoonup \overline{u} \text{ in } L^2(0,T; L^2(\Omega)) \cap L^2(0,T; H^1_{loc}(\Omega)). 
\end{equation}
Given a sufficiently small $\eps >0$, let $ \sigma_\eps \in C([0,1])$ be the piecewise affine function such that
$$ \sigma^\eps = 1 \text{ on } [2 \eps, 1- 2 \eps], \quad \sigma^\eps =0 \text{ on } [0, \eps ] \cup [1-\eps,1]  $$
and $\sigma^\eps$ is affine on $[\eps, 2 \eps]$ and on $[1- 2\eps, 1 - \eps]$. For any $\varphi \in \Phi^0_0$, we set $\varphi^\eps(t,x', x_N) = \sigma^\eps(x_N) \varphi(t,x', x_N)$. Clearly $\varphi^\eps \in \Phi_\eps$ so we can use it as a test function in \eqref{s2-weak}. Since $b_\eps=1$ on the support of $\varphi^\eps$  \eqref{s2-weak} reduces to 
\begin{equation} \label{weak_deg}
 \int_\Omega u^\eps_0 \varphi^\eps(0) + \int_0^T\!\! \int_\Omega u^\eps \varphi^\eps_t  + \int_0^T \!\! \int_\Omega a_\eps \nabla u^\eps \cdot \nabla \varphi^\eps = 0 .
\end{equation} 
By \eqref{init_conv_ND} and strong convergence $\varphi^\eps(0) \to \varphi(0)$ in $L^2(\Omega)$, 
$$\int_\Omega u^\eps_0 \varphi^\eps(0) \to \int_\Omega u_0 \varphi(0) \quad \text{as } \eps \to 0^+.$$
Similarly, by \eqref{deg_conv} and strong convergence $\varphi^\eps_t \to \varphi_t$ in $L^2(0,T;L^2(\Omega)),$
$$\int_0^T\!\! \int_\Omega u^{\eps_k} \varphi^{\eps_k}_t \to \int_0^T\!\! \int_\Omega \overline{u} \varphi_t .$$ 
Finally, 
$$\int_0^T \!\! \int_\Omega a_{\eps} \nabla u^{\eps} \cdot \nabla \varphi^\eps = \int_0^T \!\! \int_\Omega a_\eps \sigma^\eps \nabla u^\eps \cdot \nabla \varphi + \int_0^T \!\! \int_\Omega a_\eps u^\eps_{x_N}  \varphi \sigma^\eps_{x_N}.$$
Taking into account locally uniform convergence $\sigma^\eps \to 1$ in $(0,1)$ and inequality $\sigma^\eps \leq 1$, we show 
$$\int_0^T \!\! \int_\Omega a_{\eps_k} \sigma^{\eps_k} \nabla u^{\eps_k} \cdot \nabla \varphi \to \int_0^T \!\! \int_\Omega \nabla \overline{u} \cdot \nabla \varphi$$
by the same reasoning as in the proof of Theorem \ref{thm_conv_dbc}. 
On the other hand, 
$$\left|\int_0^T \!\! \int_\Omega a_\eps u^\eps_{x_N}  \varphi \sigma^\eps_{x_N}\right| \leq \sqrt{\int_0^T \!\! \int_\Omega a_\eps |\nabla u^\eps|^2} \cdot \eps^{-1} \sqrt{\max_{[\eps, 2 \eps]} a_\eps} \sqrt{\int_0^T \!\! \int_{\Omega_{\eps} \setminus \Omega_{2 \eps}} \varphi^2}  $$
which converges to $0$ as $\eps \to 0^+$ by virtue of \eqref{energy_eq_bl} and \eqref{asa_deg}. 
Summing up, passing to the limit $\eps_k \to 0^+$ in \eqref{weak_deg}, we obtain 
\begin{equation*} 
 \int_\Omega u_0 \varphi(0) + \int_0^T\!\! \int_\Omega \overline{u} \varphi_t  + \int_0^T \!\! \int_\Omega \nabla \overline{u} \cdot \nabla \varphi = 0 
\end{equation*} 
for any $\varphi \in \Phi_0^0$, i.e. $\overline{u}$ coincides with the weak solution $u$ to \eqref{eqN} with initial datum $u_0$. By the usual argument involving uniqueness, we deduce that the whole sequence $u^\eps$ converges to $u$.   
\end{proof} 

\section{Reilly-type identity}\label{s-Re}

We establish an estimate for  solutions of the following elliptic problem, where region $\Omega$ is the flat cylinder defined in (\ref{def-om}),
\begin{equation}\label{rnie-f}
 \di (a_\varepsilon \nabla u) = f, \qquad \hbox{in }\Omega,
\end{equation}
where $a_\varepsilon$ is defined in (\ref{dea}), i.e., 
$$
a_\eps(x) := \min\{ \frac {d(x)}\eps, 1\}.
$$
Hence $a_\eps$ satisfies (\ref{asa_nondeg}) for $p=1$, see Remark \ref{pot-a}. We assume that $f\in L^2(\Omega)$. We use 
the energy functional $E_\eps$ 
on $L^2_\eps(\Omega)$ by  formula (\ref{defE})

The point is that the boundary conditions in (\ref{rnie-f}) are not explicitly specified. The main result of this section is as follows.
\begin{theorem}\label{lem1.8}
If $\eps>0$ is fixed, $u, \di (a_\varepsilon \nabla u) \in L^2(\Omega)$,  $E_\varepsilon(u)<\infty$ 
and eq. (\ref{rnie-f}) is satisfied as the equality of $L^2$ functions, 
then
\begin{equation}\label{reilly}
 \int_\Omega |\di (a_\eps\nabla u)|^2 = \int_\Omega \sum_{i,j}^N a_\eps^2 
\left|\frac{\partial^2 u}{\partial x_i\partial x_j}\right|^2 \,dx + 
\frac 1{\varepsilon}\int_{\partial \Omega_\varepsilon}  \left|\frac{\partial u}{\partial x_N} \right|^2\,d\cH^{N-1}.
\end{equation}
\end{theorem}
We stress that a part of our motivation stems from the desire to show existence of a non-trivial and bounded normal derivative.
Another motivation to look for such a result, apart from differential geometry, see \cite{Re}, is the study of singularities of the Laplace equation in polygonal/polyhedral domains, see e.g. \cite[Theorem 2.2.1]{grisvard}. 

Let us notice that Theorem \ref{lem1.8} implies that for each $\varepsilon$ the mixed derivatives $u_{x_ix_j}$ are in $L^2(\Omega, a_\eps^2\cL^N)$. However, we do not have any estimates uniform in $\varepsilon$.

Before we prove this result, we will establish  a series of lemmas.

\begin{lemma}\label{lem1.1}
Let us suppose that $\Omega\subset \bR^N$ is defined in (\ref{def-om}) and $\varepsilon>0$ is fixed. In addition, $u\in L^2(\Omega)$ is such that:\\
(1) $E_\varepsilon(u) <\infty$; \\
(2) 
$\di (a_\varepsilon \nabla u) \in L^2(\Omega)$.\\
Then, the trace $a_\varepsilon \frac{\partial u}{\partial \nu}$ on $\partial\Omega$ exists and it is zero. Here, $\nu$ is the outer normal to $\Omega.$
\end{lemma}
\begin{proof}
Since  $\di (a_\varepsilon \nabla u) \in L^2(\Omega)$, then we deduce that $u\in H^2_{loc}(\Omega).$ Indeed, for any point $x_0\in \Omega$, we take any cut-off function $\vfi\in C^\infty_c(B_\delta(x_0))$, where $B_\delta(x_0)\subset \Omega.$ Then, the standard regularity theory yields that $u\vfi \in H^2(\Omega)\cap H^1_0(\Omega).$

In addition, since  $\di (a_\varepsilon \nabla u) \in L^2(\Omega)$, the classical theory of Fujiwara-Morimoto, cf. \cite{fujiwara}, implies existence of the trace of the normal component of $a_\varepsilon \nabla u $. It is well-defined as an element of $(H^{1/2}(\partial\Omega))^*$.

Let us take $\delta>0$ and set
$$
\Gamma_\delta = \{ x\in \Omega: d(x) = \delta\}.
$$
For our choice of $\Omega$, made in (\ref{def-om}), $\Gamma_\delta$ is always smooth. 
We are going to exploit  the fact that the outer normal to $\Omega_\delta$, $\nu$, does not depend on $\delta$ (for $\delta<\frac12$) and it equals the outer normal to $\partial\Omega.$ In this case, we notice that the trace of $a_\epsilon \frac{\partial u}{\partial\nu}$ exists and it is a 
$H^{1/2}(\Gamma_\delta)$ function, because
$u \in H^{2}_{loc}(\Omega).$ 
Moreover,
\begin{equation}\label{eq01.12}
\lim_{\delta\to 0^+}    \int_{\Gamma_\delta} 
\gamma^\delta \left(a_\epsilon  \frac{\partial u}{\partial\nu}\right) \gamma^\delta(\varphi)\, d\cH^{N-1}
= \langle \gamma 
\left( a_\epsilon  \frac{\partial u}{\partial\nu}\right),  \gamma 
(\varphi)\rangle, \qquad \varphi \in H^{1}(\Omega).
\end{equation}
Here, $\gamma^\delta\equiv \gamma_{\partial\Omega_\delta}$ denotes the trace operator, $\gamma^\delta: H^{1}(\Omega_{\delta}) \to H^{1/2}(\Gamma_\delta)$, see Section \ref{sec2}.

Indeed, by  definition the LHS takes the form,
$$
\lim_{\delta\to 0^+}   \int_{\Gamma_\delta} \gamma^\delta(a_\epsilon  \frac{\partial u}{\partial\nu}) \gamma^\delta \varphi\, d\cH^{N-1}
= \lim_{\delta\to 0^+}  \int_{\Omega_\delta} \di(  \varphi a_\epsilon  \nabla u)\, dx
=  \int_{\Omega} \di(  \varphi a_\epsilon  \nabla u)\, dx.
$$
We notice here that the convergence of the RHS is 
due to 
our assumption (2).

Let us suppose that contrary to our claim 
$$
\gamma(a_\varepsilon \frac{\partial u}{\partial \nu}) \neq 0.
$$
This means that there is $\varphi\in H^{1}(\Omega)$ such that 
$\langle a_\epsilon  \frac{\partial u}{\partial\nu}, \gamma(\varphi)\rangle \neq 0.$ We could even assume that $\varphi \ge0$. Indeed, if $a^+ = \max\{ a, 0\}$ and $a^- = (-a)^+$, then $\varphi = \varphi^+ - \varphi^-$  and we know that $\varphi^\pm \in  H^{1}(\Omega)$. Since $\langle a_\epsilon  \frac{\partial u}{\partial\nu},  \varphi\rangle \neq 0$ we deduce that $\langle a_\epsilon  \frac{\partial u}{\partial\nu},  \varphi^+\rangle \neq 0$ or $\langle a_\epsilon  \frac{\partial u}{\partial\nu},  \varphi^-\rangle \neq 0$.
In other words, there exists $\eta>0$ and $0\le\varphi \in H^{1}(\Omega)$ such that
\begin{equation}\label{eq0112-2}
0< \eta \le \int_{\Gamma_\delta} \gamma^\delta (a_\epsilon  \frac{\partial u}{\partial\nu}) \gamma^\delta (\varphi)\, d\cH^{N-1},
\end{equation}
for all $\delta \in (0, \delta_0).$ This claim follows directly from (\ref{eq01.12}).

We may assume that $\varphi$ in (\ref{eq0112-2}) is smooth due to the density of smooth functions in $H^1(\Omega)$. 
In particular, this implies that $\varphi(x)\in[0,K]$ for all $x\in \Omega.$

Let us consider $A_+=\{ x\in\Omega:  \ a_\eps\frac{\partial u}{\partial \nu}(x) \ge0\}. $ Then, (\ref{eq0112-2}) and $\varphi(x) \le K$ imply that 
\begin{equation}\label{eq-2.5p}
\eta \le \left(\int_{\Gamma_\delta\cap A_+} +
\int_{\Gamma_\delta\setminus  A_+}\right)
\gamma^\delta (a_\epsilon  \frac{\partial u}{\partial\nu}) \gamma^\delta (\varphi)\, d\cH^{N-1} \le K\int_{\Gamma_\delta\cap A_+}
\gamma^\delta (a_\epsilon  \frac{\partial u}{\partial\nu})
\, d\cH^{N-1}.
\end{equation}
Let us integrate both sides of (\ref{eq-2.5p}) over $(0,\tilde\delta)$ with respect to $x_N$. In this way we obtain,
$$
\frac{\eta\tilde\delta}K \le \int_{A_+\setminus \Omega_{\tilde\delta}} 
a_\eps \frac{\partial u}{\partial \nu} \,dx.
$$
In order to estimate the RHS, we 
use the Cauchy-Schwarz inequality,  which yields,
$$
\frac{\eta\tilde\delta}K \le 
\left(\int_{A_+\setminus \Omega_{\tilde\delta}} 
a_\eps \,dx \right)^{1/2}
\left(\int_{A_+\setminus \Omega_{\tilde\delta}} a_\eps\left|\frac{\partial u}{\partial \nu}\right|^2 \,dx \right)^{1/2} \le \frac{\tilde\delta}{(2\eps)^{1/2}}
\left(\int_{A_+\setminus \Omega_{\tilde\delta}}a_\eps\left|\frac{\partial u}{\partial \nu}\right|^2 \,dx \right)^{1/2} .
$$
After cancelling $\tilde\delta$ on both sides, we see that the RHS goes to zero as $\tilde\delta\to0$, while the LHS remains bounded away from zero. This contradiction proves our claim. 
\end{proof}

Now, we  are ready for {\it the proof of Theorem \ref{lem1.8}.} In fact, it is inspired by results like \cite[Theorem 2.2.1]{grisvard}.
Due to the special structure of $\Omega$, we have that 
$$
a_\eps(x) =  
\left\{
\begin{array}{ll}
  \frac{x_N}\varepsilon  &   x_N \in (0,\varepsilon), \\
   1  & x_N \in [\varepsilon, 1- \varepsilon]), \\
   \frac{1-x_N}{\varepsilon}   & x_N \in (1-\varepsilon, 1),
\end{array}
\right. \qquad \nabla a_\eps (x) \equiv e_N s(x) = e_N
\left\{
\begin{array}{ll}
\frac 1\varepsilon &   x_N \in (0,\varepsilon), \\
 0  & x_N \in [\varepsilon, 1- \varepsilon]), \\
   -\frac{1}{\varepsilon}   & x_N \in (1-\varepsilon, 1).
\end{array}
\right.
$$
At this moment, we make an additional smoothness assumption on $u$, namely $u\in C^\infty(\overline\Omega)$. Later we will relax it. We obviously have,
$$
\int_\Omega |\di (a_\eps\nabla u)|^2 = \int_\Omega ((a_\eps \Delta u )^2
+ 2 a_\eps s \frac{\partial u}{\partial x_N} \Delta u  + 
  |s|^2\left|\frac{\partial u}{\partial x_N} \right|^2) =
  I_1 +I_2 + \frac 1{\varepsilon^2}\int_{\Omega\setminus \Omega_\varepsilon} \left|\frac{\partial u}{\partial x_N} \right|^2.
$$
We  have to transform $I_1,$
$$
I_1 = \int_\Omega \sum_{i,j=1}^N a_\eps^2 \frac{\partial^2 u}{\partial x_i^2}
\frac{\partial^2 u}{\partial x_j^2} =: \sum_{i,j=1}^N J_{ij}.
$$
We will inspect each $J_{ij}$, when $i$ and $j$ are smaller than $N$. Since we assumed high regularity of $u$, we may integrate  $J_{ij}$ by parts twice,
$$
J_{ij} = - \int_\Omega \frac{\partial}{\partial x_j}\left( a_\eps^2 \frac{\partial^2 u}{\partial x_i^2} \right) \frac{\partial u}{\partial x_j} =
- \int_\Omega a_\eps^2 \frac{\partial^3 u}{\partial x_i^2 \partial x_j} \frac{\partial u}{\partial x_j} =
\int_\Omega a_\eps^2 \frac{\partial^2 u}{\partial x_i \partial x_j} \frac{\partial^2 u}{\partial x_j\partial x_i}.
$$
Here we used the lack of boundary terms and the fact that $a_\eps$ commutes with $\frac{\partial }{\partial x_j} $ for $j<N$. We see that $J_{ij}$ has the desired form for $i,j<N$. Of course, $J_{NN}$ has a sign too. We have to look at the remaining terms $J_{iN}$ for $i<N$. We notice,
\begin{eqnarray*}
J_{iN} &=&
\int_\Omega a_\eps^2(x_N) \frac{\partial^2 u}{\partial x_i^2} 
\frac{\partial^2 u}{\partial x_N^2} =
- \int_\Omega a_\eps^2(x_N)  \frac{\partial u}{\partial x_i} 
\frac{\partial^3 u}{\partial x_N^2 \partial x_i}\\
&= &
 \int_\Omega \frac{\partial }{\partial x_N}
\left( a_\eps^2(x_N) \frac{\partial u}{\partial x_i}\right) \frac{\partial^2 u}{\partial x_i \partial x_N} -
\int_{\partial \Omega } a_\eps^2(x_N)  \frac{\partial^2 u}{\partial x_i \partial x_N} \frac{\partial u}{\partial x_i}.
\end{eqnarray*}
Of course, the boundary term vanishes for smooth and bounded functions, and this is the case we are considering now. 
We continue,
$$
J_{iN}
= \int_\Omega 
a_\eps^2(x_N) \left|\frac{\partial^2 u}{\partial x_i \partial x_N}\right|^2+
\int_\Omega   2 a_\eps(x_N) s(x_N) \frac{\partial u}{\partial x_i }
\frac{\partial^2 u}{\partial x_N\partial x_i }.
$$
We will combine the last term with $I_2$,
$$
I_2 = 2\sum_{i=1}^N \int_\Omega a_\eps(x_N)s(x_N) \frac{\partial u}{\partial x_N }
\frac{\partial^2 u}{\partial x_i^2}=: 2\sum_{i=1}^N K_i.
$$
We integrate by parts, keeping in mind that for $i< N$ we have,
$$
K_i =-  \int_\Omega a_\eps(x_N) s(x_N)
\frac{\partial^2 u}{\partial x_N\partial x_i }\frac{\partial u}{\partial x_i }.
$$
Finally,
$$
\sum_{i=1}^{N-1} J_{iN} + I_2 = 
\sum_{i=1}^{N-1} \int_\Omega 
a_\eps^2 \left|\frac{\partial^2 u}{\partial x_i \partial x_N}\right|^2 
+ 2 K_N .
$$
We have to investigate $K_N$. We notice that in fact we integrate over $\Omega \setminus \Omega_\varepsilon $, then the integration by parts yields,
\begin{eqnarray*}
K_N & = & \frac12 \int_{\Omega\setminus \Omega_\varepsilon} a_\eps(x_N)s(x_N)  \frac{\partial}{\partial x_N}\left(
 \frac{\partial u}{\partial x_N} \right)^2 
 = \frac12 \left(\int_0^\varepsilon + \int_{1-\varepsilon}^1\right)
 \int_{\bT^{N-1}}  a_\eps(x_N)s(x_N) \frac{\partial}{\partial x_N}\left(
 \frac{\partial u}{\partial x_N} \right)^2 \\
 &=& - \frac12\int_{\Omega\setminus \Omega_\varepsilon} \frac{\partial}{\partial x_N}
 (a_\eps(x_N)s(x_N)) \left(
 \frac{\partial u}{\partial x_N} \right)^2 \\
&&
 + \frac12 \int_{\bT^{N-1}}(a_\eps(x_N)s(x_N))\left(
 \frac{\partial u}{\partial x_N} \right)^2 \left|_{x_N=0}^{x_N=\varepsilon}\right. + \frac12\int_{\bT^{N-1}}(a_\eps(x_N)s(x_N))\left(
 \frac{\partial u}{\partial x_N} \right)^2 \left|_{x_N=1 - \varepsilon}^{x_N=1}\right.\\
&=&  - \frac 12\int_{\Omega\setminus \Omega_\varepsilon} 
  s^2(x_N)\left|\frac{\partial u}{\partial x_N}\right|^2
  + \frac1{2\varepsilon}\int_{\bT^{N-1}}\left|\frac{\partial u}{\partial x_N}\right|^2(x',\varepsilon)
  + \frac1{2\varepsilon}\int_{\bT^{N-1}}\left|\frac{\partial u}{\partial x_N}\right|^2(x',1-\varepsilon).
\end{eqnarray*}
Here we use the assumption that $u$ is smooth and bounded,  hence the boundary term vanishes and we obtain,
$$
2 K_N = - \frac1{\varepsilon^2} \int_{\Omega\setminus \Omega_\varepsilon}  \left| \frac{\partial u}{\partial x_N} \right|^2
+ \frac1{\varepsilon}\int_{\partial\Omega_\varepsilon } \left|\frac{\partial u}{\partial x_N}\right|^2.
$$
As a result,
$$
\int_\Omega | \di(a_\eps \nabla u)|^2 =
\sum_{i,j=1}^N \int_\Omega a_\eps^2 \left| \frac{\partial^2 u}{\partial x_i \partial x_j}\right|^2 +\frac1{\varepsilon}\int_{\partial\Omega_\varepsilon }
\left|\frac{\partial u}{\partial x_N}\right|^2.
$$

Now, we have to relax the regularity assumption on $u$. We have to exercise a bit of care due to the presence of a weight which vanishes at the boundary of $\Omega$. We will use the fact that $u\in H^{2}_{loc}(\Omega)$ to our advantage. We will use the general approach with necessary modifications. We set 
$$
U_n := \Omega_{n+2}\setminus \overline\Omega_n.
$$
Of course, $U_n \cap U_k = \emptyset$, when $|n-k|>4$. The family of sets $\{U_n\}_{n=1}^\infty$ forms an open covering of the set $\Omega$. We may find a smooth partition of unity subordinate to this covering, i.e. $\{\varphi_n\}_{n=1}^\infty$ such that $\varphi_n\in C^\infty_c(\Omega)$, $\supp \varphi_n \subset  U_n$ and $\sum_{n=1}^\infty \varphi_n = 1$. 

Now, for a fixed $\eta>0$ and all natural $n$ we can find $v_n\in C^\infty_c(U_{n+1} \cup U_n \cup U_{n-1})$ 
and 
$$
\| \varphi_n u - v_n \|_{H^{2}(U_{n+1} \cup U_n \cup U_{n-1})} 
\le \frac \eta{ 2^{n+1}}.
$$
We set 
$$
v^\eta = \sum_{n=1}^\infty v_n.
$$
Since the sum is locally finite, we conclude that $v$ is smooth and in $H^{2}_{loc}(\Omega)$. We claim that
\begin{align}\label{r122}
 &| E(u) - E(v^\eta)| \le \rho,\nonumber\\
 &\left| \| \di (a_\eps\nabla u)\|_{L^2} - \| \di (a_\eps\nabla v^\eta)\|_{L^2} \right| \le \rho,\nonumber\\
 &\left| \int_\Omega a_\eps^2 \left| \frac{\partial^2 u}{\partial x_i \partial x_j}\right|^2 -  \int_\Omega a_\eps^2 \left| \frac{\partial^2 v^\eta}{\partial x_i \partial x_j}\right|^2  \right|\le \rho,\quad \ i,j\in\{1,\ldots, N\},\\
&  \frac1{\varepsilon}\left|  \int_{\partial\Omega_\varepsilon }
\left|\frac{\partial u}{\partial x_N}\right|^2  -  \int_{\partial\Omega_\varepsilon }
\left|\frac{\partial v^\eta}{\partial x_N}\right|^2  \right|\le \rho.\nonumber
\end{align}
All these expressions have the same structure so it is enough to investigate one of them in detail. For this reason we look at the second estimate. We have,
\begin{equation}\label{r12}
\left|\int_\Omega | \di(a_\eps \nabla u)|^2  - \int_\Omega | \di(a_\eps \nabla v^\eta)|^2\right| \le \| \di( a_\eps \nabla (u - v^\eta))\|_{L^2}
  \| \di( a_\eps \nabla (u + v^\eta))\|_{L^2} . 
\end{equation}
We investigate at the  first factor on the RHS
\begin{align*}
\| \di( a_\eps \nabla (u - v^\eta))\|_{L^2} & \le
\sqrt2 \|a_\eps\Delta(u - v^\eta) \|_{L^2} + \sqrt2 \| \nabla a_\eps \|_{L^\infty} \|\nabla u -\nabla v^\eta\|_{L^2}\\
&\le \sqrt2(1 + \| \nabla a_\eps \|_{L^\infty}) \sum_{n=1}^\infty \|\varphi_n u - v_n\|_{H^{2}} \le
\sqrt2(1 + \| \nabla a_\eps \|_{L^\infty}) \eta.
\end{align*}
The second factor of the RHS of (\ref{r12}) we can  deal with in a similar way,
\begin{align*}
 \| \di( a_\eps \nabla (u + v^\eta))\|_{L^2} & = \| \di( a_\eps \nabla (2u + (v^\eta- u)))\|_{L^2} \le 2 \| \di( a_\eps \nabla u)\|_{L^2} +
 \| \di( a_\eps \nabla (u - v^\eta))\|_{L^2}.
\end{align*}
Combining these observation we see,
$$
 \left|\int_\Omega | \di(a_\eps \nabla u)|^2  - \int_\Omega | \di(a_\eps \nabla v^\eta)|^2\right| \le
  \eta\sqrt2(1 + \| \nabla a_\eps \|_{L^\infty})\left(
   2 \| \di( a_\eps \nabla u)\|_{L^2} + \eta\sqrt2(1 + \| \nabla a_\eps \|_{L^\infty})\right).
$$   
Hence, our claim (\ref{r122}) holds. With its help we immediately deduce that (\ref{reilly}) follows under the theorem assumptions. $\qed$



\section{
Derivation of the dynamic boundary conditions}

In this section, we state the convergence problem for strong solutions and develop the necessary tools. In particular, we introduce a generalization of the notion of the $\Gamma$-convergence here.

\subsection{The problem statement}
We derive, in this section, the dynamic boundary conditions as a limit of
\begin{equation}\label{eq1p}
\begin{array}{ll}
b_\varepsilon u^\varepsilon_t = \di (a_\varepsilon \nabla u^\eps)     & (x,t)\in \Omega_T , \\
 u^\eps(x,0) = u_0^\eps(x)    & x\in \Omega,
\end{array}
\end{equation}
where $b_\eps$ is defined in (\ref{def-be}) and $a_\eps$ is given by (\ref{dea}). We stress the fact that due to Lemma \ref{lem1.1} the degenerating weight $a_\eps$ makes the above eq. well-posed without any explicit boundary conditions.

The limit passage depends on the value of $\kappa$, see its definition in (\ref{df-KA}).
In this section, we assume that $\kappa\in (0,\infty)$.

Our goal is achieved in a few steps, starting from a priori estimates on solutions through the $\Gamma$-limit computations and finishing with
a  derivation of the Energy-Dissipation Balance (EDB for short) for (\ref{eq2}).

We might say that the above eq. develops a boundary layer, because the height of $b_\eps$ grows as $\eps\to0$. The case  $\kappa \in(0,\infty)$ is special because we will prove with the help of EDB that the limit eq. of (\ref{eq1p}) as $\eps\to0$ is the heat eq. with dynamic boundary condition involving the parameter $\kappa$.

We first state the existence result for (\ref{eq1p}).  We will use the K\=omura theory of nonlinear semigroups for this purpose, see \cite{brezis}. We are going to write  (\ref{eq1p}) as
\begin{equation}\label{eq2-CA}
\left\{    \begin{array}{l}
      - u_t \in \partial_{L^2_\eps} E_\varepsilon(u),   \\
       u(0) = u_0.
   \end{array}
\right.   
\end{equation}
Here, $ \partial_{L^2_\eps} E_\varepsilon(u) $ is the subdifferential of $E_\eps$ with respect to the $L^2_\eps$ metric, i.e.,
$$
\partial_{L^2_\eps} E_\varepsilon(u) =
\{\xi\in L^2_\eps(\Omega):\ \forall h\in L^2_\eps(\Omega)\  
E_\eps(u+h) - E_\eps(u) \ge \langle \xi, h \rangle_{L^2_\eps} \}. 
$$
We have to compute the subdifferential in  $L^2_\eps$.
\begin{lemma} If $E_\eps$ is defined over $L^2_\eps(\Omega)$ by formula (\ref{defE}), then\\
(a) $\partial{L^2_\eps} E_\varepsilon(u) = \{-\frac 1{b_\varepsilon}\di (a_\varepsilon \nabla u) \} $,
and  $D(\partial E_\varepsilon) = \{u\in D(E_\varepsilon): \ \hbox{div}\, (a_\epsilon\nabla u) \in L^2(\Omega)\} $.\\
(b)  $H^2(\Omega) \subset D(\partial E_\varepsilon)$.
\end{lemma}
\begin{proof}
It is easy to deduce from the definition of the subdifferential that  $\zeta\in \partial_{L^2_\eps}E_\eps(u)$ if and only if for all $h\in L^2_\eps(\Omega)$ we have,
$$
\int_\Omega a_\eps\nabla u^\eps \nabla h \, dx = 
\int_\Omega b_\eps \zeta h\, dx.
$$
This identity means that the weak divergence of $a_\eps\nabla u^\eps$ exists and equals to $-b_\eps \zeta$. This observation proves part (a). Part (b) is obvious.
\end{proof}
This lemma implies that (\ref{eq1p}) is  the gradient flow of $E_\epsilon$ with respect to the inner product of $L^2_\varepsilon.$  It also gives us the following basic existence result following the classical K\=omura theory, here we follow Brezis exposition, see \cite[Theorem 3.2]{brezis}. 

\begin{proposition}
Let us suppose that $E_\eps$ is defined in (\ref{defE}), $u_0 \in D( E_\varepsilon)$, then there is a unique $u^\eps \in H^{1}_{loc}([0,\infty);  L_\eps^2(\Omega)$, which is a
solution to (\ref{eq2-CA}).
Moreover, the function $[0,\infty)\ni t\mapsto E_\eps(u(t))$ is decreasing and for a.e. $t>0$ we have $u(t)\in  D(\partial E_\varepsilon))$.
\end{proposition}
\begin{proof}
This result is a consequence of the convexity and lower semicontinuity of $E_\eps$ combined with \cite[Theorem 3.2]{brezis}. 
\end{proof}
We notice that (\ref{eq2-CA}) is nothing else but (\ref{eq1p}). Keeping this in mind, we present the main result of this section about convergence when $\varepsilon\to0$.

\begin{theorem}\label{thmain} 
Let us suppose that $\kappa\in(0,\infty)$ and $u^\eps$ is a unique solution of the gradient flow (\ref{eq2}), when $a_\eps$ is defined in (\ref{dea}),
the initial condition $u_0^\eps$ are in $D(\partial E_\varepsilon)$, $u_0^\eps\to u_0$ in $L^2(\Omega)$ and
$$
\sup_{\varepsilon>0} E_\varepsilon(u_0^\eps), \quad \sup_{\varepsilon>0} 
\| u_0^\eps\|_{L^2_\varepsilon} <\infty.
$$
\noindent
(1) If $\varepsilon\to 0$, then $u^\eps$ converges to $u\in H^2(\Omega)$, in the sense of $H^2_{loc}(\Omega)$ convergence. In particular, 
$u^\varepsilon \chi_{\Omega_\varepsilon} \to u$ in $L^2$, $\nabla u^\varepsilon \chi_{\Omega_\varepsilon} \to \nabla u$ in $L^2$, 
$\Delta u^\varepsilon \chi_{\Omega_\varepsilon} \to \Delta u$ in $L^2$ 
and $m_\eps(u^\eps)_t$ converges weakly in $(H^{1/2}(\partial\Omega))^*$ to $-\gamma(\frac{\partial u}{\partial \nu})$. Moreover, $u$ is a unique
solution of 
\begin{equation}\label{eql}
  \begin{array}{ll}
      u_t  = \Delta u & (x,t)\in \Omega \times (0,T),   \\
      u_t + \frac1\kappa\frac{\partial u}{\partial \nu}  =0  & (x,t)\in \partial\Omega \times (0,T),   \\
      u(x,0) = u_0(x) & x\in \Omega .
   \end{array}  
\end{equation}
(2) Eq. (\ref{eql}) is the gradient flow of $E_0$ in $X_0^\kappa$, where both objects are defined below, 
$$
E_0(u, v) = \left\{
\begin{array}{ll}
\frac12\int_\Omega | \nabla u|^2     &  \hbox{if } \gamma(u) =v, 
\ u\in H^1(\Omega),
\\
+\infty     & \hbox{otherwise.}
\end{array}
\right.
$$
and we set $X^\kappa_0= H^\kappa_0$, where $H^\kappa_0$ is defined in (\ref{df-Xka}).
\end{theorem}
The question of $\Gamma$-convergence of $E_\eps$ is interestinig.
Actually, if we stick to the $L^2$-metric, then we can show that $E_\eps$ $\Gamma$-converges in the $L^2$-topology to $\bar E$ on $L^2$, defined as 
$$
\bar E(u) = \left\{
\begin{array}{ll}
\frac12\int_\Omega | \nabla u|^2     &  \hbox{if } u \in H^1(\Omega),\\
+\infty     & u \in L^2(\Omega) \setminus H^1(\Omega),
\end{array}
\right.
$$
see Lemma \ref{lemE0}.

We stress that $\bar E$ is defined on a smaller space than $E_0,$ hence these two functionals are different. It turns out that a new notion of $\Gamma$-convergence with respect to $\iota_\eps$, extending the classical one, is necessary. We introduce it in Definition \ref{dGC} and we show that functionals $E_\eps$  $\Gamma$-converge with respect to $\iota_\eps$ to $E_0$. This is a result complementary to Theorem \ref{thmain}.

The new notion of $\Gamma$-convergence permits us to deal with changing underlying spaces and their topologies. Here, we do not explore this notion fully. 

Our main tool in the proof of  this theorem is  the Energy-Dissipation Balance. 
For this purpose we will collect below a series of estimates and separately a series of the $\Gamma$-convergence results.

\subsection{Estimates on solutions}

Here, we will derive  a series of  energy estimates keeping in mind the general form of $b_\eps$, see (\ref{def-be}). 
We begin with a basic one.
\begin{lemma}\label{best}
If  $b_\eps$ is defined by (\ref{def-be}),  $u^\eps$ is a solution to (\ref{eq1p}), $\sup_\varepsilon \|u_0^\eps\|_{L^2_\varepsilon}\le M<\infty,$ then 
\begin{equation}\label{eq5}
   \hbox{\rm ess}\sup_{t\in(0,T)}\left( \int_{\Omega} b_\varepsilon (u^\eps)^2(x,t) \,dx + \int_0^TE_\varepsilon(u(t))\,dt\right)\le M<\infty.
\end{equation}
\end{lemma}
\begin{proof}
We multiply (\ref{eq1p}) by $u^\eps$ and integrate over $\Omega_T.$ This yields,
$$
\frac12 \frac{d}{dt}\int_{\Omega_T} b_\varepsilon u^2_\varepsilon\,dxdt = 
\int_{\Omega_T} u \di (a_\varepsilon \nabla u)\,dxdt.
$$
The LHS is easy to compute. We integrate the RHS by parts using Lemma \ref{lem1.1}. Hence, we integration over $[0,T]$ yields
\begin{equation}\label{eq6}
\frac12\int_{\Omega} b_\varepsilon u^2(x,T)\,dx +  \int_{\Omega_T} a_\varepsilon |\nabla u(x,t)|^2\,dxdt
= \frac12\int_{\Omega} b_\varepsilon (u^\eps_0)^2(x)\,dx \le M<\infty.
\end{equation}
\end{proof}

\begin{lemma}\label{leEDB}
Let us suppose that 
$u$ is a solution to the gradient flow (\ref{eq2}). Then,
\begin{equation}\label{edb}
    \frac12 \int_{\Omega_T} b_\varepsilon u^2_t\,dx dt + 
    \frac12 \int_{\Omega_T} \frac 1{b_\varepsilon}|\di (a_\varepsilon\nabla u)|^2\,dx dt +
    E_\varepsilon(u(T)) = E_\varepsilon(u_0^\eps).
\end{equation} 
\end{lemma}
\begin{remark}
We notice that (\ref{edb}) is a form of the needed Energy-Dissipation Balance.
\end{remark}
\begin{proof}
In order to establish Lemma \ref{leEDB} we multiply (\ref{eq1p}) by $u_t,$ which belongs to $L^2(\Omega_T).$ Here, we use the regularity of the initial conditions. We obtain,
\begin{equation}\label{eq3}
    \int_{\Omega_T} b_\varepsilon u^2_t\,dx dt  = \int_{\Omega_T} \di(a_\varepsilon\nabla u ) u_t \,dxdt =:I.
\end{equation}
We want to  integrate by parts on the right-hand-side (RHS) of (\ref{eq3}).
Formally, this and  Lemma \ref{lem1.4} yield,
$$
I = -\int_{\Omega_T} a_\varepsilon \nabla u \nabla u_t\, dxdt
= - \int_0^T (\frac 12\frac{d}{dt} \int_\Omega a_\varepsilon |\nabla u|^2 \,dx)\,dt = E_\varepsilon(u_0) - E_\varepsilon(u(T)).
$$
Hence, we would obtain
\begin{equation}\label{eq44}
\int_{\Omega_T} b_\varepsilon u_t^2 + E_\eps(u(T)) = E_\eps(u_0).
\end{equation}
We will justify it below.

Let us suppose for the moment that (\ref{eq44}) holds. If so
we divide both sides of (\ref{eq1p}${}_1$) by $\sqrt{b_\varepsilon}$, then we square both sides. The result is
$$
\int_{\Omega_T} b_\varepsilon u_t^2 \,dxdt= \int_{\Omega_T} \frac 1{b_\varepsilon} |\di (a_\eps\nabla u)|^2 \,dxdt
$$
and  (\ref{edb}) follows.
\end{proof}

Now, we are going to justify (\ref{eq44}).
For this purpose we shall prove the following lemma:
\begin{lemma} \label{lem1.4}
Let us suppose that  $u\in L^2(\Omega_T)$ is such that 
$$
\int_0^T E_\varepsilon (u(\cdot, t))\,dt<\infty \qquad\hbox{and}\qquad
\int_0^T \int_{\Omega_T}| \di(a_\varepsilon\nabla u )|^2 \,dxdt <\infty
$$  
and $\eta$ is the standard mollifier kernel. Then $(u*\eta)(x,t) = \int_0^\infty u(x,t)\eta(s-t)\,ds$ is such that $(u*\eta)(\cdot,t) $ satisfies the assumptions of Lemma \ref{lem1.1} for a.e. $t\in (0,T)$.
\end{lemma}
\begin{proof}
The argument is based on the following observation. If $v\in L^2(\Omega_T)$, then $v*\eta \in L^2(\Omega_T).$
If we keep this in mind then
$$
\int_0^T E_\varepsilon ((u*\eta)(\cdot, t))\,dt = \int_0^T\int_\Omega  | (v*\eta)(x, t))|^2\,dxdt< \infty,
$$
where $v = \sqrt{a_\varepsilon}\nabla u$. Hence, the first part of the lemma follows. 

The second part is proved in the same way. 
\end{proof}

Now, we are in a position to justify (\ref{eq44}). We have
$$
\int_{\Omega_T} \di(a_\varepsilon\nabla u ) u_t \,dxdt = 
\lim_{\delta\to 0}\int_{\Omega_T} \di(a_\varepsilon\nabla u*\eta_\delta ) (u*\eta_\delta)_t \,dxdt =R,
$$
where $\eta_\delta(x) = \frac 1\delta\eta\left(\frac x\delta\right) $. We use here the fact that $u*\eta_\delta$ converges to $u$ in $L^2$, when $\delta$ goes to zero. Then, due to the above Lemmas, we have,
$$
R =
- \lim_{\delta\to 0}
\int_{\Omega_T} a_\varepsilon\nabla u*\eta_\delta \nabla  (u*\eta_\delta)_t \,dxdt = \lim_{\delta\to 0} E_\varepsilon((u*\eta_\delta)(\cdot, 0) - E_\varepsilon((u*\eta_\delta)(\cdot, T) .
$$
We have to show that 
$$
\lim_{\delta\to 0} E_\varepsilon((u*\eta_\delta)(\cdot, t) = E_\varepsilon((u)(\cdot, t).
$$
Indeed, we have
$$
\int_\Omega a_\eps(|\nabla u(t_2)|^2 - |\nabla u(t_1)|^2 )\,dxdt =
\int_\Omega a_\eps (\nabla u(t_2) + \nabla u(t_1))(\nabla u(t_2) - \nabla u(t_1)  )\,dxdt =: D
$$
The integration by parts and Lemma \ref{lem1.1} imply that
$$
D = \int_\Omega  ( u(t_2) -  u(t_1))\di (a_\eps(\nabla u(t_2) - \nabla u(t_1) ))\,dxdt .
$$
Hence,
$$
|D| \le \int_{t_1}^{t_2} \| u_t\| \,dt \cdot (\| \di (a_\eps \nabla u(t_2))\| + \di(a_\eps\nabla u(t_1) )\|)
$$
and the continuity  follows. \hfill\qed

The lemma above makes the  proof of Lemma \ref{leEDB} complete. 

Another important element of our analysis is a study of the behavior of the boundary layers as their width goes to zero. Our tool is the average.
The following Lemma is valid for any $\kappa$, however, it yields useful estimates for positive, possibly infinite, $\kappa$.

\begin{lemma}\label{lem1.6}
Let us suppose that 
$u^\eps\in L^2(\Omega_T)$ is a sequence of  solutions to (\ref{eq1p}) and $m_\eps(u^\eps) \in L^2(\Sigma_T)$, where $m_\eps(u^\eps)$ is given by (\ref{def-we}). Here, we introduce the shorthand,
$\Sigma_T = \partial\Omega\times(0,T).$ Then,\\
(a) the sequences $u^\eps$ and $u^\varepsilon_t$ are bounded in $L^2_\varepsilon(\Omega_T)$;\\
(b)  the sequences $m_\eps(u^\eps),$ $m_\eps(u^\eps)_t$ are bounded in $L^2(\Sigma_T)$;\\
(c) $u^\eps\xrightharpoonup{*} u$ in $L^\infty(0,T; L^2(\Omega))$.

In particular (a) and (b) imply
$$
u^\eps \rightharpoonup u \qquad\hbox{in } L^2(\Omega_T),\qquad\qquad
u^\varepsilon_t \rightharpoonup u_t \qquad\hbox{in } L^2(\Omega_T),
$$
\begin{equation}\label{eq7}
m_\eps(u^\eps) \rightharpoonup w \qquad\hbox{in } L^2(\Sigma_T),\qquad\qquad
m_\eps(u^\eps)_t \rightharpoonup w_t \qquad\hbox{in } L^2(\Sigma_T).
\end{equation}
\end{lemma}
\begin{proof}
The validity of the first part follows from Lemma \ref{best} and Lemma \ref{leEDB}. Thus, we obtain
$$
M \ge \int_{\Omega_T} b_\varepsilon (u^\eps)^2 \,dx\ge \int_{\Omega_T}  (u^\eps)^2 \,dx
$$
and
$$
M \ge \int_{\Omega_T} b_\varepsilon (u^\varepsilon_t)^2 \,dx\ge \int_{\Omega_T}  (u^\varepsilon_t)^2 \,dx,
$$
because $b_\varepsilon\ge 1.$ Hence,  (\ref{eq7}${}_1$) follows.

We notice that $m_\eps(u^\eps)$ are uniformly bounded in $ L^2(\Sigma_T)$. Indeed, for small $\eps$, we have,
$$
\| m_\eps(u^\eps) \|^2 = 
\int_0^T\int_{\partial\Omega} \left(\frac1\varepsilon\int_0^\varepsilon u^\eps\, dx_N \right)^2\,d x' 
\le \frac 1{\varepsilon \phi(\eps)}
\int_0^T \int_{\Omega\setminus \Omega_\varepsilon} b_\eps (u^\eps)^2 \,dxdt 
\le \frac M\kappa +1,
$$
where $M$ is the constant appearing in (\ref{eq5}). The same argument applies to $w_{\varepsilon,t},$ but we have to invoke Lemma \ref{leEDB} and (\ref{edb}). Hence, we can deduce (\ref{eq7}${}_2$). 

The last part follows immediately from (\ref{eq6}).
\end{proof}
The point is that the estimate on the norm of $m_\eps$ depends also on $\kappa$. We shall see that when $0<\kappa$ and it is finite, then the estimates yield the dynamic boundary condition. When $\kappa$ is infinite, then $M(\eps)\to0$, when $\eps\to0$ and we will end up with zero Dirichlet boundary values in the limit. We will deal with this case 
in next Section.

We will need information about the behavior of $u^\eps$ in $\Omega_T$.
\begin{lemma}\label{lem1.7}
Let us suppose that 
$u^\eps\in L^2(\Omega_T)$ is a sequence of  solutions to (\ref{eq1p}) and the assumptions of Theorem \ref{thmain} hold. Then, there exists $u\in L^2(0,T;H^1)$ such that $\Delta u \in L^2(\Omega)$ and for all $\delta>0$
$$
u^\eps\rightharpoonup u  \qquad\hbox{in }  L^2(0,T;H^1(\Omega_\delta)),\qquad\qquad
\Delta u^\eps \rightharpoonup \Delta u \qquad\hbox{in }L^2(\Omega_\delta\times(0,T)).
$$
\end{lemma}
\begin{proof}
If $\varepsilon<\delta$, then (\ref{eq6})  yields,
$$
\|\nabla u^\eps \|_{L^2(\Omega_\delta\times(0,T)} \le M,
$$
where $\nabla = \nabla_x$ denotes the spacial gradient.
Thus, the first part of our claim follows.

By a similar token 
(\ref{edb}) implies that
$$
M \ge \int_0^T\int_{\Omega_\delta} |\Delta u^\eps |^2 \,dxdt.
$$
Moreover, Theorem \ref{lem1.8} implies that $u^\eps\in H^2_{loc}(\Omega).$ Our claim follows.
\end{proof}
We might say that the above result is a version of Lemma \ref{traceU}, which is integrated in time.

We have gathered enough information to state a key observation. It is a part of Theorem \ref{thmain}, however due to its importance it is worth stating separately.
\begin{theorem}\label{trach}
If $u$ is the weak limit of $u^\varepsilon,$ then for a.e. $t>0$ we have
$$
(\gamma(u))_t = -\frac1\kappa\gamma \left(\frac{\partial u }{\partial \nu} \right).
$$ 
The equality above is in the sense of elements of $(H^{\frac12}(\partial\Omega))^*$.
\end{theorem}
\begin{proof}Since $a_\eps$ defined by (\ref{dea}) satisfied (\ref{asa_nondeg}) with $p=1$, then by Lemma \ref{traceU} and Lemma \ref{lem11} $m_\eps(u^\eps)$ converges to $\gamma (u)$
in $L^2(\partial \Omega)$ for a.e. $t>0$. Due to Lemma \ref{lem1.6} we also know that $m_\eps(u^\eps)_t$ converges weakly in $L^2(\Sigma_T)$ to a limit, called $w_t$. Thus, we deduce that $\gamma(u)$ is in $H^{1}(0,T; L^2(\partial\Omega))$ and  $m_\eps(u^\eps)_t \rightharpoonup \gamma(u)_t$.
 
Now, we will show that $m_\eps(u^\eps)_t$ converges weakly$^*$ to $-\frac1\kappa\gamma(\frac{\partial u}{\partial x_N})$.
Indeed, if we take $\varphi\in H^1(\Omega)$, then we have
$$
\int_{\partial\Omega} m_\eps(u^\eps)_t \gamma\varphi\, d\cH^{N-1} 
= \frac1{\eps\phi(\eps)}\int_{\Omega\setminus\Omega_\eps}b_\varepsilon u^\eps_t
\varphi\, dx + \frac1{\eps\phi(\eps)}\int_{\Omega\setminus\Omega_\eps}b_\varepsilon u^\eps_t(\varphi - \gamma\varphi)\,dx =:I + E.
$$
It is easy to check that $E$ goes to zero when $\eps\to0$, provided that $\kappa\in (0,\infty).$

We notice that the integration by  parts yields,
\begin{eqnarray*}
I&=& \frac1{\eps\phi(\eps)}\int_{\Omega\setminus\Omega_\eps}
\vfi \di(a_\varepsilon \nabla u^\varepsilon)  \, dx =
\frac{-1}{\eps\phi(\eps)}\int_{\Omega\setminus\Omega_\eps}  a_\varepsilon \nabla u^\varepsilon \nabla \varphi \, d\cH^{N-1}dx_N - \frac1{\eps\phi(\eps)}
 \int_{\partial\Omega_\eps} 
 \frac{\partial u^\varepsilon}{\partial x_N}
\gamma^\eps  \varphi
 \, d\cH^{N-1}\\
&=: &J_\varepsilon + B_\varepsilon.
\end{eqnarray*}
Here, we used that fact that due to Lemma \ref{lem1.1} the integral over $\partial\Omega$
vanishes. The definition of the traces yields $\gamma^\eps \varphi\to \gamma \varphi$ in $L^2(\partial\Omega)$. Moreover, due to 
Lemma \ref{traceU} (2) we infer that
$$
\lim_{\eps\to0} B_\eps = - \frac 1\kappa\int_{\partial\Omega}
 \frac{\partial u}{\partial x_N}
\gamma\varphi\,d\cH^{N-1} .
$$
We notice that
$$
\eps^2\phi^2(\eps)\left| J_\varepsilon \right|^2 \le 
\int_{\Omega\setminus \Omega_\varepsilon } 
a_\varepsilon |\nabla u^\varepsilon|^2\,dx 
\int_{\Omega\setminus \Omega_\varepsilon } 
a_\eps |\nabla \varphi|^2\,d\cH^{N-1} 
\, dx \le 2 E_\varepsilon(\nabla u^\varepsilon) \|\nabla\varphi\|^2_{L^2(\Omega\setminus \Omega_\varepsilon )}.
$$
We see that the RHS goes to zero, when $\varepsilon\to 0$.
\end{proof}

\subsection{$\Gamma$-limits} \label{s-G-lim}
Once we derived de Giorgi's energy-dissipation balance (EDB), (\ref{edb}), we want to use it 
to pass to the limit with the gradient flows (\ref{eq1p}) or equivalently (\ref{eq2-CA}). We hope that this process will lead to  an EDB for the limit, hence we will be able to discover the equation governing the limit.

The limit process requires computing  the $\Gamma$-limit of $E_\eps$, i.e. the functional generating the gradient flow. Lemma \ref{lemE0} below is quite straightforward, we compute $\bar E = \Gamma-\lim_\eps E_\eps{}$. However, we shall see later, that $\bar E$ does not generate the gradient flow that we obtain after the limit passage.

The problem becomes more visible, when we analyze functional 
\begin{equation}\label{def-Fe}
F_\varepsilon(u) = \frac12\int_0^T\int_{\Omega} b_\varepsilon u^2_t(x,t)\,dx.
\end{equation}
In this case 
Lemma \ref{lem11} shows the effect of the boundary layer induced by the weight $b_\eps$. We also see how important is the way we define $b_\eps$. Actually, this Lemma suggests that $F_0$, the limit of $F_\eps$, is defined on a bigger space than each individual $F_\eps$. For this purpose we will modify the notion of $\Gamma$-convergence, see below. Our first task is computing the $\Gamma$-limit of $E_\eps$.


\begin{lemma}\label{lemE0} If $a_\eps$ is defined in (\ref{dea}), $E_\eps$ is given by (\ref{defE}), then
$$
\Gamma-\lim_{\varepsilon\to 0} E_\varepsilon = \bar E 
\qquad\hbox{in the $L^2(\Omega)$-topology,}
$$
where
\begin{equation}\label{defE0}
\bar E(u) =   \left\{ \begin{array}{ll}
    \frac12\int_\Omega  |\nabla u |^2\,dx   &  u \in H^1(\Omega),\\
    +\infty     & u \in L^2(\Omega)\setminus H^1(\Omega).
    \end{array}
    \right.
\end{equation}
\end{lemma}
\begin{proof}
First, we prove the lower semicontinuity part: if $u^\epsilon \to u_0$ in $L^2(\Omega)$, then
$$
L:=\varliminf_{\varepsilon\to 0}\int_\Omega a_\varepsilon |\nabla u^\eps|^2 \,dx\ge
\int_\Omega  |\nabla u_0|^2 \,dx.
$$
We recall that $\Omega_\delta = \{ x\in\Omega: d(x,\partial\Omega)>\delta\}$. We can find $\eps_0>0$ such that for all $n\in\bN$ and $1/n<\eps_0$ and $\varepsilon< 1/n$ we have
$$
\int_{\Omega_{1/n}}|\nabla u^\eps|^2 \le L+1.
$$  
Hence, we can select a weakly converging  subsequence (without relabeling) 
$$
u^\eps \rightharpoonup u^n\qquad\hbox{in } H^1(\Omega_{1/n}).
$$
Due to the diagonal argument there is $u_0\in H^1(\Omega),$ it is such that 
$$
u^\eps \rightharpoonup u_0\qquad\hbox{in } H^1(\Omega_{1/n})
$$
for $\varepsilon< 1/n.$ In fact, since we assumed that $u^\eps \to u_0$ in $L^2(\Omega)$, we conclude that the convergence in each $H^1(\Omega_{1/n})$ takes place without the need to extract subsequences. Hence, the claim follows.

Finding a recovery sequence is easy, because for all $\varepsilon>0$ we have $H^1(\Omega)\subsetneq D(E_\varepsilon)$. So, if $u_0\in H^1(\Omega)$ is given, then the necessary recovery sequence is constant, $u^\eps\equiv u_0$. Indeed,
$$
\lim_{\varepsilon\to0} E_\varepsilon(u_0) =
\lim_{\varepsilon\to0} \frac12\int_\Omega a_\varepsilon|\nabla u_0|^2 = \frac12\int_\Omega |\nabla u_0|^2 = \bar E 
(u_0).
$$
\end{proof} 

A more interesting development is related to functionals involving the weight $b_\eps,$ which are independent of $a_\eps$. Moreover, Lemma \ref{lem11} suggests that in the limit we will see the behavior on the boundary. In other words, in the limit the underlying space becomes bigger. In order to capture the emerging boundary behavior, we introduce a a modified notion of $\Gamma$-convergence. It is an extension of this notion to the case, when the underlying space changes with the parameter $\eps>0.$

Let $E_\varepsilon$ be defined in $D(E_\varepsilon)$ which is a subset of a set $X_\varepsilon$, which is a topological vector space for $\varepsilon\in[0,1)$.
As we shall see that the actual meaning of $X_\eps$ will change according to our needs.

Let $\iota_\varepsilon$ be a mapping from $X_\varepsilon$ to $X_0$ for $\varepsilon>0$. We assume that $X_0$ is a topological space.
We  generalize the notion of $\Gamma$-convergence as follows.
\begin{definition} \label{dGC}
We say that $E_\varepsilon$ $\Gamma$-converges to $E_0$ with respect to $\iota_\varepsilon$ if
\begin{enumerate}
\item[(i)] (lower semicontinuity) $E_0(w)\leq\liminf_{\varepsilon\downarrow 0} E_\varepsilon(
u^\eps)$ if 
$\iota_\varepsilon(u^\eps) \to w$
\item[(ii)] (recovery sequence) for any $w\in X_0$, there is $
u^\eps$ 
such that
\[
	E_0(w) = \lim_{\varepsilon\downarrow 0} E_\varepsilon (u^\eps)
	\quad\text{and}\quad \iota_\varepsilon u^\eps \to w.
\]
\end{enumerate}
We simply write $E_\varepsilon \xrightarrow{\Gamma} E_0$ in $\iota_\varepsilon$.
\end{definition}

Here, we are interested in the case when the spaces depend on time. We consider $X_\eps^T = L^2(0,T; X_\eps)$, 
where $X_\eps\equiv L^2_\eps(\Omega)$ and $L^2_\eps(\Omega)$ was defined in (\ref{df-Le}). 
For $\kappa\in(0,\infty)$,
we set $X_0^{T,\kappa} =L^2(0,T; X_0^\kappa)$, where
$X_0^\kappa$ is defined in (\ref{df-Xka})%
For $\eps\in(0,1)$, we define $\iota_\eps: X_\eps^T\to X_0^{T,\kappa}$ as follows,
$$
\iota_\eps (v) = (v,  m_\eps v),
$$
where the operator $m_\eps$ is defined in (\ref{def-we}).

With this extension of the notion of the $\Gamma$-convergence we establish a natural analogue of Lemma \ref{lemE0}:
\begin{lemma}\label{lem1.10} Let us suppose that $u\in H^{1}(0,T; L^2_\eps(\Omega))$ and functionals $F_\eps$ are defined in (\ref{def-Fe}) with $D(F_\eps) =  H^{1}(0,T; L^2_\eps(\Omega))$. If $X_\eps$, $\iota_\eps$ are defined above and  $\kappa\in(0,\infty)$, 
then
$$
\Gamma-\lim_{\varepsilon\to 0^+}  F_\varepsilon =
F_{bs}\qquad\hbox{with respect to }\iota_\eps,
$$
where
$$
F_{bs}(u,w) = 
\left\{ 
\begin{array}{ll}
 \int_0^T\int_\Omega u^2_t\,dx + \kappa\int_0^T\int_{\partial\Omega} w_t^2\,d\cH^{N-1}, & (u,w)\in   H^{1}(0,T; L^2_\eps(\Omega)\times L^2_\eps(\partial\Omega) )\\
  +\infty   & otherwise.
\end{array}
\right.
$$
\end{lemma}
\begin{proof}
We first establish the lower semicontinuity. For this purpose we take $u^\varepsilon\in H^{1}(0,T; L^2_\eps(\Omega))$. We study 
$\displaystyle{
\liminf_{\varepsilon\to 0} F_\varepsilon(u^\eps)}$.
Let us fix $\delta> 0$ and consider $\eps<\delta$, then we have
$$
F_\varepsilon(u) \ge \int_0^T\int_{\Omega_\delta} (u^\varepsilon_t)^2\,dxdt + 
\phi(\eps) \eps
 \int_0^T\int_{\partial\Omega} (m_\eps(u^\eps)_t)^2\,d\cH^{n-1}dt,
$$
This is so, because
$$
\int_{\partial\Omega} (m_\eps(u^\eps)_t)^2\,d\cH^{n-1} \le \frac1\varepsilon 
\int_{
\Omega\setminus\Omega_\eps} 
(u^\varepsilon_t)^2\,dx_Nd\cH^{n-1}.
$$
Thus, the lower semicontinuity of the norm and (\ref{eq7}) imply that for all $\delta>0$ we have
$$
\liminf_{\varepsilon\to 0} F_\varepsilon(u^\eps) \ge  
\int_0^T\int_{\Omega_\delta} u_t^2\,dxdt + \kappa
 \int_0^T\int_{\partial\Omega} w_t^2\,d\cH^{n-1}dt.
$$
Sending $\delta$ to 0 yields
$$
\liminf_{\varepsilon\to 0} F_\varepsilon(u^\eps) \ge  F_{bs}(u,w).
$$
Now, for a given pair 
$(u, w) \in H^{1}(0,T; L^2(\Omega))\times  H^{1}(0,T; L^2(\partial\Omega))$ we have to construct a recovery sequence, $u^\varepsilon\in  H^{1}(0,T; L^2(\Omega)).$ Namely, for all $\varepsilon\in(0,1)$, we 
set
$$
u^\varepsilon = \chi_{\Omega_\varepsilon} u + \chi_{\Omega\setminus\Omega_\varepsilon} w.
$$
Of course $u^\varepsilon\in H^{1}(0,T; L^2_\eps(\Omega))$. We also have
\begin{multline*}
F_\varepsilon(u^\varepsilon) =
\int_0^T\int_{\Omega_\varepsilon} u^2_t\,dxdt + \int_0^T\int_{\Omega\setminus\Omega_\varepsilon} b^\varepsilon w^2_t\,d\cH^{n-1}dt\\ =
\int_0^T\int_{\Omega_\varepsilon} u^2_t\,dxdt + \phi(\eps) \eps
\int_0^T \int_{\partial\Omega}\int_0^\varepsilon \frac1\varepsilon w^2_t \,dx_N d\cH^{n-1}dt.
\end{multline*} 
Since  the boundary of $\Omega$  is flat, we deduce that
$$
\int_{\partial\Omega}\int_0^\varepsilon \frac1\varepsilon w^2_t \,dx_N d\cH^{n-1}
= \int_{\partial\Omega} w^2_t \, d\cH^{n-1}.
$$
Then,
$$
\lim_{\varepsilon\to 0} F_\varepsilon(u^\varepsilon) = 
\int_0^T\int_{\Omega} u^2_t\,dxdt + \kappa\int_0^T \int_{\partial\Omega} w^2_t \, d\cH^{n-1}.
$$
Hence, our claim follows.
\end{proof}

The last result in the spirit of this subsection, that we need, is a good lower estimate on $\liminf$ of functional $I_\eps$ defined below. This could hardly be called 
$\Gamma$-$\liminf$ in $\iota_\eps$, but it 
will be sufficient for our purposes.
\begin{lemma}\label{lem1.11}
Let us assume that $b_\varepsilon$ is defined by (\ref{def-be}) with $\kappa\in(0,\infty)$. We set
$$
I_\varepsilon(u) = \int_{\Omega_T} \frac{1}{b_\varepsilon}(\div (a_\varepsilon\nabla u))^2\,dx
\qquad\hbox{\rm for }u\in D(\partial E_\varepsilon).
$$
If $u^\eps \to u$ in $L^2$, then 
$$
\liminf_{\eps\to 0} I_\eps(u^\eps) \ge 
\int_{\Omega_T} |\Delta u|^2\, dx + \frac1\kappa
\int_0^T\int_{\partial\Omega} \left(\frac{\partial u}{\partial \nu}\right)^2\,d\cH^{N-1}  =: I_0(u).
$$
\end{lemma}
\begin{proof}
We have two ways of looking at $I_\eps$. The first one is just recalling the
definition of $a_\varepsilon$. Then, we have
$$\di(a_\eps\nabla u ) = \frac1\varepsilon \frac{\partial u}{\partial \nu}\chi_{\Omega\setminus\Omega_\varepsilon} +
a_\eps \Delta u.$$
Hence,$$
I_\varepsilon(u) = \frac1{\phi(\varepsilon)}\int_0^T \int_{\Omega}
( \frac1\varepsilon \frac{\partial u}{\partial \nu}\chi_{\Omega\setminus\Omega_\varepsilon}  + a_\eps \Delta u)^2 +
\left(1-\frac1{\phi(\varepsilon)}\right)
 \int_0^T \int_{\Omega_\varepsilon} | \Delta u|^2.
$$
The argument we used in Lemma \ref{lem1.7}  implies that
$$
\Delta u ^\varepsilon \chi_{\Omega_\varepsilon}  \rightharpoonup \Delta u
$$
in $L^2(\Omega)$. Hence,
$$
\varliminf_{\varepsilon\to 0} \int_0^T\int_{\Omega_\eps}|\Delta u^\varepsilon|^2
 \ge \int_{\Omega_T} | \Delta u|^2.
$$
In particular, this implies that $\frac{\partial u }{\partial\nu}$ is well-defined as an element of $(H^{1/2}(\partial\Omega))^*$ due to the classical theory; see e.g.\ \cite{fujiwara}, where $L^p$ theory is discussed. 

Another way of looking at $I_\eps$ exploits the Reilly identity, Theorem \ref{lem1.8}, in an essential way, 
in this case we have,
\begin{multline*}
\int_{\Omega} \frac{1}{b_\varepsilon}(\div (a_\varepsilon\nabla u))^2\, dx \geq \frac1{\phi(\eps)}\int_\Omega |\di (a_\eps \nabla u^\eps)|^2\, dx\\
= \frac1{\phi(\eps)}\int_\Omega \sum_{i,j}^N a_\eps^2 
\left|\frac{\partial^2 u^\eps}{\partial x_i\partial x_j}\right|^2 \,dx + 
\frac{1}{\phi(\eps)\eps}\int_{\partial \Omega_\varepsilon}  \left|\frac{\partial u^\eps}{\partial x_N} \right|^2\,d\cH^{N-1}
\ge\frac1{\phi(\eps)\eps}\int_{\partial \Omega_\varepsilon}  \left|\frac{\partial u^\eps}{\partial x_N} \right|^2\,d\cH^{N-1}.    
\end{multline*}
We see that we can use Lemma \ref{traceU} to estimate the RHS above. Now, we combine these computations to obtain,
\begin{align*}
\liminf_{\eps\to 0} I_\eps(u^\eps) & \ge
\liminf_{\eps\to 0} (1-\frac1{\phi(\varepsilon)})I_\eps(u^\eps)+ \liminf_{\eps\to 0} \frac1{\phi(\varepsilon)} I_\eps(u^\eps) \ge \int_{\Omega_T} |\Delta  u|^2\,dxdt + \frac1\kappa \int_{\Sigma_T} 
\left|\frac{\partial u}{\partial x_N} \right|^2\,d\cH^{N-1}dt.
\end{align*}
\end{proof}

\subsection{The Energy-Dissipation Balance, proof of Theorem \ref{thmain}}


After  having gathered a series of statements on $\Gamma$-convergence, including Lemma \ref{lemE0}, we may  apply the de Giorgi--Serfaty--Mielke theory, see \cite{sylvia}, \cite{M}. 
Here, for the sake of the clarity of the argument, we assume that the data are well-prepared, i.e. $\sup_{\varepsilon>0}(E_\varepsilon(u_0)+ \|u_0\|_{L^2_\varepsilon})<\infty.$

We note that our assumptions imply that $u^\varepsilon$ converges to $u$ locally in the $H^2$-norm and
$u^\varepsilon \chi_{\Omega_\varepsilon} \to u$ in $L^2$, $\nabla u^\varepsilon \chi_{\Omega_\varepsilon} \to \nabla u$ in $L^2$, 
$\Delta u^\varepsilon \chi_{\Omega_\varepsilon} \to \Delta u$ in $L^2$. Moreover, the key statement on traces, i.e. $m_\eps(u^\eps)_t$ converges to $-\gamma(\frac{\partial u}{\partial \nu})$ weakly$^*$ in $(H^{1/2}(\partial\Omega))^* $ is proved in Theorem \ref{trach}. 

Now, we want to study 
the limit passage in (\ref{edb}). If we succeed, then we will have
another limiting Energy-Dissipation Balance,
\begin{equation}\label{edb2}
\frac 12\int_\Omega |\nabla u(x,T)|^2\, dx +\frac12\int_{\Omega_T}(u_t^2 + |\Delta u|^2)\,dxdt + 
\frac12\int_{\Sigma_T}(\kappa (\gamma(u)_t)^2 + \frac1\kappa(\frac{\partial u}{\partial \nu})^2)\, d\cH^{N-1}dt = \frac 12\int_\Omega |\nabla u_0|^2\, dx.
\end{equation}
Then, the proof of the main theorem will  follow from 
(\ref{edb2}) by differentiation in time. 

In order to obtain (\ref{edb2}) we follow (\cite{sylvia}) and we apply the liminf to (\ref{edb}),
$$
\liminf_{\varepsilon\to0} (E_\varepsilon(u_0^\eps) - E_\varepsilon(u^\varepsilon(T))) =
\liminf_{\varepsilon\to0} \frac12 \int_{\Omega_T} b_\varepsilon u^2_t\,dx dt + 
    \frac12 \int_{\Omega_T} \frac 1{b_\varepsilon}|\di (a_\varepsilon\nabla u)|^2\,dx dt.
$$
We may now apply Lemma \ref{lem1.10} and Lemma \ref{lem1.11} to deduce 
$$
\liminf_{\varepsilon\to0} (E_\varepsilon(u_0^\eps) - E_\varepsilon(u^\varepsilon(T))) \ge
\frac12\int_{\Omega_T}(u_t^2 + |\Delta u|^2)\,dxdt + 
\frac12\int_{\Sigma_T}(\kappa w_t^2 + \frac1\kappa(\frac{\partial u}{\partial \nu})^2)\, d\cH^{N-1}dt.
$$
We recall that $w$ is the weak limit in $L^2(\partial\Omega)$ of $m_\eps(u^\eps)$ and 
$m_\eps(u^\eps)_t \in L^2(\partial\Omega\times(0,T))$. Since $m_\eps(u^\eps)_t$ has a weak limit, see (\ref{eq7}), we deduce that $w_t$ is in  $L^2(\partial\Omega\times(0,T))$, so the formula above is correct. Moreover, we recall that
Lemma \ref{traceU} and Lemma \ref{lem11} give  us $w = \gamma(u)$.
Hence $w_t = \gamma(u)_t$. As a result, we may rewrite the above inequality in a  way that is
more suitable for our purposes,
$$
\liminf_{\varepsilon\to0} (E_\varepsilon(u_0^\eps) - E_\varepsilon(u^\varepsilon(T))) \ge
\frac12\int_{\Omega_T}(u_t^2 + |\Delta u|^2)\,dxdt + 
\frac12\int_{\Sigma_T}( \kappa\gamma(u)_t^2 + \frac1\kappa(\frac{\partial u}{\partial \nu})^2)\, d\cH^{N-1}dt =:R.
$$
Of course, we have
$$
\frac12\int_{\Omega_T}(u_t^2 + |\Delta u|^2)\,dxdt \ge 
\int_{\Omega_T}u_t \Delta u \,dxdt .
$$
We wish to integrate by parts, but apparently $u_t$ lacks the necessary regularity. For this purpose we mollify $u$ by convolving with the product $\varphi (x',x_N) = \varphi_{N-1}(x')\varphi_1(x_N)$ of  standard kernels, where $\varphi_{N-1}$ (resp. $\varphi_1$) depends on $N-1$ (resp. one) variables so that $\gamma (\varphi^\delta* u) = \gamma(u)* \varphi_{N-1}^\delta$. In order to make the notation short we will write $u^\delta$ for the regularization so that $\gamma(u^\delta_t) = \gamma(u)_t^\delta$.
Thus, 
\begin{eqnarray*}
R & =&\lim_{\delta\to 0} \int_{\Omega_T}u^\delta_t \Delta u^\delta \,dxdt +
\frac12\int_{\Sigma_T}( \kappa(\gamma(u)^\delta_t)^2 + \frac1\kappa(\frac{\partial u^\delta}{\partial \nu})^2)\, d\cH^{N-1}dt\\
&= &\lim_{\delta\to 0} [- \int_{\Omega_T} \nabla u^\delta_t \nabla u^\delta \,dxdt
+ \int_{\Sigma_T}(\gamma(u_t)^\delta \frac{\partial u^\delta}{\partial \nu}+
 \frac{\kappa}2(\gamma(u)^\delta_t)^2 + \frac1{2\kappa}(\frac{\partial u^\delta}{\partial \nu})^2)
\, d\cH^{N-1}dt]
\\
&\ge &E(u_0) - E(u(T)).
\end{eqnarray*}
Since $\lim_{\varepsilon\to0}E_\varepsilon(u_0) = E(u_0)$ we deduce that for any $T>0$ we have
$$
E(u(T)) \le \liminf_{\varepsilon\to0} E_\varepsilon (u^\varepsilon(T)) \le 
\limsup_{\varepsilon\to0} E_\varepsilon (u^\varepsilon(T)) \le E(u(T)).
$$
If we combine these observation we will reach the identity (\ref{edb2}). Let us now differentiate (\ref{edb2}) with respect to time. As a result, we obtain
$$
0 = \int_\Omega \nabla u \nabla u_t + \frac12\int_\Omega( u_t^2 + \Delta u^2)\,dxdt+
\frac12\int_{\partial\Omega}(\kappa(\gamma(u)_t)^2 + \frac1\kappa\gamma( \frac{\partial u}{\partial \nu})^2)\, d\cH^{N-1}
$$
for a.e. $t>0$. The first term on the RHS is the action of $\nabla u_t\in (H^1)^*$ over $\nabla u \in H^1$. By definition  it is
$$
\int_\Omega \nabla u \nabla u_t = - \int_\Omega \Delta u u_t \,dx +
\int_{\partial\Omega}\gamma(u)_t \gamma( \frac{\partial u}{\partial \nu}))\, d\cH^{N-1}.
$$
The last integral on the RHS is well-defined, because we showed above that $\gamma(u)_t \in L^2(\Sigma_T)$. Moreover, we know that $\gamma( \frac{\partial u}{\partial \nu})\in L^2(\Sigma_T)$ too.
Hence,
$$
0 = \frac12  \int_\Omega ( \Delta u -u_t )^2\,dx + 
\frac12 \int_{\partial\Omega}(\sqrt\kappa\gamma(u)_t +\frac1{\sqrt\kappa}\gamma( \frac{\partial u}{\partial \nu}))^2\, d\cH^{N-1}.
$$
Equivalently, this means that (\ref{eql}) holds.

Let us now  show uniqueness of solutions to  (\ref{eql}). Suppose that we have two solutions $w_1, w_2$ of this equation. Then, $u:= w_2 - w_1$ satisfies  (\ref{eql}) with zero initial conditions. We may test this eq. with  $u$. As a result we obtain,
$$
\frac 12 \frac d{dt}\int_\Omega u^2\, dx = \int_\Omega u \Delta u \,dx.
$$
We may integrate the RHS by parts. Then, we obtain,
$$
\frac 12 \frac d{dt}\int_\Omega u^2\, dx = -\int_\Omega |\nabla u|^2\,dx
+\int_{\partial\Omega } \gamma u \gamma (\frac{\partial u }{\partial \nu})\, d\cH^{N-1}.
$$
Now, (\ref{eql}${}_2$) implies that,
$$
\frac 12 \frac d{dt}\int_\Omega u^2\, dx +
\frac 12 \frac d{dt}\int_{\partial\Omega} (\gamma(u))^2\, dx 
= -\int_\Omega |\nabla u|^2\,dx\le 0.
$$
Thus, the first claim follows.

We will establish the second part of the theorem. For this purpose we have to compute the differential of $E_0$ in $X_0^\kappa$. We notice that if $(u_0, v_0)\in D(E_0)$, then $u_0\in H^1(\Omega)$ and $v_0 = \gamma(u_0)$. By definition of the subdifferential, if $(u_0, \gamma(u_0))\in D(\partial E_0)$, then there is $(\xi, \zeta) \in L^2(\Omega)\times L^2(\partial\Omega)$ that for all $(h_1, h_2) \in L^2(\Omega)\times L^2(\partial\Omega)$ we have
$$
E(u_0+h_1,\gamma(u_0) + h_2) - E(u_0,\gamma(u_0)) \ge
((\xi,\zeta), (h_1, h_2))_{X_0^\kappa}.
$$
However, the LHS is finite if and only if $h_1\in H^1(\Omega)$
and $h_2 = \gamma(h_1).$
It is also easy to see that 
$$
\int_\Omega \nabla u_0\nabla h_1\, dx = \int_\Omega \xi h_1\, dx + 
\kappa\int_{\partial\Omega} \zeta \gamma(h_1)\, d\cH^{N-1} .
$$
This  in turn implies that $\xi = - \Delta u_0$ and $\zeta =\frac1\kappa\frac{\partial u_0}{\partial \nu}$.
\qed

We have to clarify the relationship between $E_\eps$ and $E_0$. We have already seen that the $\Gamma$-limit of $E_\eps$ with respect to $L^2$-topology, $\bar E$, is incorrect. The question is if the convergence takes place in an appropriate sense. In fact we were forced to look at this issue when we considered convergence of functionals $F_\eps$, see (\ref{def-Fe}). We will use here the $\Gamma$-convergence in $\iota_\eps$ adjusted to our specific needs here.

Now, we are interested in the case when, $X_\eps = L^2_\eps(\Omega)$.
For $\eps\in(0,1)$ we define $\iota_\eps: X_\eps\to X_0^\kappa$ as follows,
$$
\iota_\eps (v) = (v,  m_\eps v).
$$
We shall see that even though $E_\eps$ does not $\Gamma$-converge to $\bar E$, then we have an appropriate replacement.
\begin{lemma}\label{l3.10}
$E_\eps$ defined in (\ref{defE}) $\Gamma$-converges to $E_0$ defined in Theorem \ref{thmain} with respect to $\iota_\eps$.
\end{lemma}
\begin{proof}
Let us take   $v^\eps\in X_\eps$, we may assume that 
\begin{equation}\label{34r1}
\liminf_{\eps\to0} E_\eps(v^\eps) <\infty,
\end{equation}
for otherwise there is nothing to prove.

Since we assumed that $\iota_\eps v^\eps \to v=(v_1, v_2)$ in $X_0^\kappa$, then we have,
$$
\| v^\eps - v_1 \|^2_{L^2(\Omega)} + \kappa
\| \bar v^\eps - v_2  \|^2_{L^2(\partial\Omega)} \to 0.
$$
We will first see that 
\begin{equation}\label{34r2}
\liminf_{\eps\to0} E_\eps(v^\eps) \ge 
\frac12\int_\Omega |\nabla v_1|^2\,dx.
\end{equation} 
Indeed, for any $\delta\in (0,\frac12)$, we have
$$
\liminf_{\eps\to0} E_\eps(v^\eps) \ge \liminf_{\eps\to0} \int_{\Omega_\delta} \frac 12 |\nabla v^\eps|^2\,dx \ge 
\int_{\Omega_\delta} \frac 12 |\nabla v_1|^2\,dx ,
$$
where we use the lower semicontinuity of the $L^2$ norm. After taking the limit above as $\delta\to 0$, we reach (\ref{34r2}).

Our assumptions permit us to invoke Lemma \ref{lem11}. This tells us that
$\lim_{\eps\to0} (m_\eps v^\eps- \gamma^\eps(u^\eps)) =0$ in the $L^2(\partial\Omega)$-norm, where $\gamma^\eps$ is defined in Lemma \ref{traceU}. At the same time Lemma  \ref{traceU} implies that $\gamma(v) = \lim_{\eps\to 0} \gamma^\eps(u^\eps)$ in  $L^2(\partial\Omega)$. Thus, we conclude that
$\gamma(v) = v_2$. As a result,
$E_0(v) = E_0(v_1, \gamma(v_1))$ is finite. The proof of (i) in Definition \ref{dGC} is finished.

If we are given $w =(v, \gamma v)\in X_0$, then for the recovery sequence, we may take $v^\eps = w.$ In this case $\iota_\eps v^\eps \to w$ due to Lemmas \ref{lem11} and \ref{traceU}.
\end{proof}
We also state a result which makes our  study of weak solutions complete.
\begin{corollary}\label{cor31} Let us define 
$E:L^2(\Omega)\to \bR\cup\{ +\infty\}$ by 
$$
E(u) = \left\{
\begin{array}{ll}
\frac12\int_\Omega |\nabla u|^2 & u\in H^1(\Omega),\\
+\infty & u\in L^2(\Omega) \setminus H^1(\Omega).
\end{array}
\right.
$$
Then, functional $E$   $\Gamma$-converges to $E_0$ with respect to $\iota_\eps,$ where $\iota_\eps$ is defined above.
\end{corollary}
\begin{proof}
Essentially, the argument is as above and it gets simpler, since (\ref{34r1}), implies that $v^\eps$ converges weakly in $H^1$ to $v$. Hence, we may use the continuity of the trace with respect to weak convergence in $H^1$. The rest of the argument goes unchanged.
\end{proof}

\section{Cases $\kappa \in \{0,\infty\}$}

Our study in the previous section did not include the cases $\kappa =0$ nor $\kappa = \infty$. We conduct the analysis of these cases here, we recall that $\kappa$ was defined in (\ref{df-KA}). We shall see that the left out cases lead in the limit to either Dirichlet or Neumann data and the argument is similar despite apparent differences.

\subsection{Case $\kappa  =\infty$: zero Dirichlet data}\label{ss-D}

We claim that in this case the influence of the boundary layer is so strong that it damps any evolution on the boundary.
Our starting point is the following inequality implied by (\ref{eq6}),
\begin{equation}\label{c3.5}
\int_{\Omega} b_\varepsilon (u^\eps)^2(x,t)\,dx +  \int_{\Omega_T} a_\varepsilon |\nabla u^\eps(x,t)|^2\,dxdt
\le \int_{\Omega} b_\varepsilon (u^\eps_0)^2(x)\,dx \le M<\infty.
\end{equation}
This means that we assume that the initial condition is well-prepared. At the same time the regularity of initial conditions implies that Lemma \ref{lem1.6} part (i) and Lemma \ref{lem1.7} hold. 

We gathered so much information about $u^\eps$ that we may show:
\begin{corollary}\label{cor4.1} Let us suppose that  $\kappa =\infty$,
$u_0^\eps$ are such that 
(\ref{c3.5}) 
holds and $\sup_{\eps>0}E_\eps(u^\eps_0)\le M<\infty.$ Then:\\
(1) there is a subsequence $u^\eps$ such that $u^\eps \xrightharpoonup{*} u$  in $L^\infty(0,T; L^2(\Omega))$ and for all $\delta>0$ $\nabla u^\eps\rightharpoonup \nabla u$ in  $L^2(0,T; L^2(\Omega_\delta))$. Moreover, $u_t,$ $\Delta u\in L^2(0,T;L^2(\Omega)).$\\
(2) $\gamma(u)=0$ in $L^2(\Sigma_T)$, i.e. $u\in L^2(0,T; H^1_0(\Omega))$;\\
(3)  $u$ is a strong solution to the heat eq.,
\begin{eqnarray}\label{r-hD}
 &u_t  = \Delta u & (x,t)\in \Omega_T,\nonumber\\
 &u = 0 & (x,t)\in \Sigma_T,\\
 &u(x,0) = u_0(x) & x\in \Omega. \nonumber
\end{eqnarray}
\end{corollary}
\begin{proof}
We may apply  Lemma \ref{lem1.6} part (i) and Lemma \ref{lem1.7} to deduce that part (1) holds. 

We have shown in (\ref{5.tr}) that
$\| m_\eps(u^\eps) \|^2 \le 
\frac1{\eps \phi} \int_\Omega b_\eps u^2_{0,\eps}\,dx \to 0,$ when $\eps \to 0$.
By Lemma \ref{traceU} we conclude that $\gamma(u) = w$ where is the weak limit of $m_\eps(u^\eps)$. However, we proved above that $m_\eps(u^\eps)$ converges strongly in $L^2(\Sigma_T)$ to 0. As a result $\gamma(u) =0$ and part (2) follows.

It remains to prove that $u$ is a solution to the heat eq. It is not clear if a version of the Energy-Dissipation Balance holds. But we will not need it. It suffices to use the weak form of equation and available estimates on strong solutions. Let us take $\varphi\in C^\infty_c(\Omega\times \bR)$ and such that $\varphi(\cdot, T) =0$. We multiply (\ref{eq2}) by $\varphi$ and integrate over $\Omega_T$. Next, we integrate by parts, 
this yields,
$$
\int_\Omega u_0^\eps \varphi \,dx=\int_{\Omega_T} (-b_\eps u^\eps \varphi_t +a_\eps \nabla u^\eps \nabla\varphi)\,dxdt =\int_0^T\int_{\Omega_\delta} (-b_\eps u^\eps \varphi_t +a_\eps \nabla u^\eps \nabla\varphi)\,dxdt .
$$
The boundary terms vanish, because the support of $\varphi$ is contained in $\Omega_\delta\times[0,T]$ for sufficiently small $\delta>0$.
We invoke Lemma \ref{lem1.7} to pass to the limit. This yields,
$$
\int_\Omega u_0 \varphi \,dx=\int_0^T\int_{\Omega_\delta} (- u \varphi_t + \nabla u\nabla\varphi)\,dxdt=\int_0^T\int_{\Omega} (- u \varphi_t + \nabla u\nabla\varphi)\,dxdt .
$$
Since $\Delta u\in L^2(0,T;L^2(\Omega))$, we may integrate by parts. Moreover, since $\varphi$ is arbitrary, we deduce that $u$ satisfies
$$
u_t = \Delta u\qquad \hbox{in } \Omega_T
$$
as well as the initial condition and the homogeneous Dirichlet boundary data.
\end{proof}

It is interesting to notice that if we are interested just in the trace of the limit $u$, then the assumption $\sup_\eps E_\eps(u^\eps_0)<\infty$ is not needed to deduce that $\gamma u = 0.$ 

\bigskip
For the sake of completeness we recall that (\ref{r-hD}) is the gradient flow of $\tilde E$, where
$$
\tilde E(u) = \left\{
\begin{array}{ll}
\int_\Omega \frac12 |\nabla u |^2\,dx& u \in H^1_0(\Omega),\\
\infty & u\in L^2(\Omega) \setminus H^1_0(\Omega).
\end{array}
\right.
$$
We should establish the relationship between $E_\eps$ and $\tilde E$. We notice that with the help of Corollary \ref{cor4.1} we could show an analogue of Lemma \ref{l3.10}.
\begin{corollary}
If $\kappa =\infty$, then $E_\eps$, defined in (\ref{defE}), $\Gamma$-converges to $\tilde E$ with respect to $\iota_\eps$. \qed
\end{corollary}

\subsection{Case $\kappa =0$: zero Neumann data}
Similarly to the previous case, 
the regularity of initial conditions imply that Lemma \ref{lem1.6} part (i) and Lemma \ref{lem1.7} hold. In particular due to Corollary \ref{cor4.1} part (1),
we have $u\in L^\infty(0,T;H^1(\Omega))$ and $\Delta u\in L^2(0,T; L^2(\Omega)).$ In particular, the trace of the normal derivative of $u$ at the boundary exists, at least as an element of $H^{-1/2}(\partial\Omega)$.
\begin{corollary}\label{cor4.2}
Let us suppose that $\kappa =0$, 
$u_0^\eps$ are such that 
$\sup_{\eps>0}E_\eps(u^\eps_0)\le M<\infty.$ Then:\\
(1) there is a subsequence $u^\eps$ such that $u^\eps \xrightharpoonup{*} u$  in $L^\infty(0,T; L^2(\Omega))$ and for all $\delta>0$ $\nabla u^\eps\rightharpoonup \nabla u$ in  $L^2(0,T; L^2(\Omega_\delta))$. Moreover, $u_t,$ $\Delta u\in L^2(0,T;L^2(\Omega)).$\\
(2) $\gamma\left(\frac{\partial u}{\partial x_n}\right)=0$ in $H^{-1/2}(\Sigma_T)$;\\
(3)  $u$ is a strong solution to the heat eq.,
\begin{eqnarray}\label{r-hN}
 &u_t  = \Delta u & (x,t)\in \Omega_T,\nonumber\\
 &\frac {\partial u}{\partial \nu} = 0 & (x,t)\in \Sigma_T,\\
 &u(x,0) = u_0(x) & x\in \Omega. \nonumber
\end{eqnarray}
\end{corollary}
\begin{proof}
Part (1) is in fact Part (1) of Corollary \ref{cor4.1}.

We recall Lemma \ref{leEDB}, it is shown for any $b_\eps$. It gives us
$$ 
\frac1{\phi(\eps)}
\int_{\Omega_T}|\di (a_\varepsilon\nabla u)|^2\,dx dt \le
\int_{\Omega_T} \frac 1{b_\varepsilon}|\di (a_\varepsilon\nabla u)|^2\,dx dt \le 2 E_\varepsilon(u_0).
$$
Combining this estimate with the Reilly identity, (\ref{reilly}), yields
\begin{equation}
\frac1{\eps\phi(\eps)} \int_{\partial\Omega_\varepsilon\times (0,T)} 
\left|\frac{\partial u}{\partial x_N} \right|^2\,\cH^{N-1}dt \le 2 E_\varepsilon(u_0).
\end{equation}
If we multiply both sides by  $\eps\phi(\eps)$, then the RHS goes to zero. We invoke the Lemma \ref{traceU} to see that 
$$
\int_0^T 
\| \gamma (\frac{\partial u}{\partial x_N} )\|^2_{L^2(\partial\Omega)}\, dt
= \lim_{\varepsilon\to 0}
\int_{\partial\Omega_\varepsilon\times (0,T)} 
\left|\frac{\partial u}{\partial x_N} \right|^2\,\cH^{N-1}dt \le 0.
$$

The comment made in Corollary \ref{cor4.1} applies: we do not have any corresponding version of the Energy-Dissipation Balance holds. For this reason,
we will use an argument similar to that used in Corollary \ref{cor4.1} to prove that  $u$ is a solution to the heat eq. We will show that $u$ is a weak solution to the heat eq. For this purpose we notice
$$
\int_{\Omega_T} (- u \varphi_t + \nabla u\nabla\varphi)\,dxdt =
\lim_{\delta\to 0} 
\int_0^T\int_{\Omega_\delta}(- u \varphi_t + \nabla u\nabla\varphi)\,dxdt =: R
$$
Now, we use the fact that $u$ is a limit, hence
\begin{align*}
R &= \lim_{\delta\to 0} \lim_{\epsilon\to 0}
\int_0^T\int_{\Omega_\delta}(- u^\eps\varphi_t + \nabla u^\eps\nabla\varphi)\,dxdt\\
& = \lim_{\delta\to 0} \lim_{\epsilon\to 0}\int_0^T\int_{\Omega_\delta\setminus \Omega_\eps}
(- u^\eps \varphi_t + \nabla u^\eps\nabla\varphi)\,dxdt + \lim_{\epsilon\to 0} \int_0^T\int_{\Omega_\eps}(- u^\eps \varphi_t + \nabla u^\eps\nabla\varphi)\,dxdt
\end{align*}
Since the sequence $u^\eps$ is uniformly bounded, then the first integral goes to zero. In the second one we integrate by parts, 
$$
R = \lim_{\epsilon\to 0}\int_0^T\int_{\Omega_\eps}(u^\eps_t - \Delta u^\eps)  \varphi\,dxdt + \int_{\Omega_\eps} u^\eps_0(x) \varphi(x,0)\,dx -
\int_0^T\int_{\partial\Omega_\eps} \frac{\partial u}{\partial x_N} \varphi \,d\cH^{N-1}dt.
$$
Since $u^\eps$ is a strong solution to (\ref{eq2}) in $\Omega_\eps\times(0,T)$, then the first term on the RHS vanishes. The second one converges to $\int_\Omega u_0(x) \varphi(x,0)\,dx$. Due to the first observation in this corollary, the third term goes to zero. Thus,
$$
\int_{\Omega_T} (- u \varphi_t + \nabla u\nabla\varphi)\,dxdt =
\int_{\Omega} u_0(x) \varphi(x,0)\,dx .
$$
Since $\varphi$ is arbitrary element of $C^\infty(\Omega_T)$ with $\varphi(\cdot, T) =0$, we deduce that $u$ satisfies (\ref{r-hN}) with homogeneous Neumann boundary condition.
\end{proof}

We close this subsection by recalling that eq. (\ref{r-hN}) is the gradient flow of $\bar E$. Moreover, we showed in Lemma \ref{l3.10} that $\bar E$ is the $\Gamma$-limit of $E_\eps$.

\subsection*{Acknowledgement}
The authors thank prof. Rodrigues for bringing paper \cite{colli} to their attention. YG was in part supported by the Japan Society for the Promotion of Science through grants No.\ 19H00689 (Kiban A), No.\ 18H05323 (Kaitaku) and by Arithmer Inc.\ and Daikin Industries Ltd.\ through a collaborative grant. This work was partly created during M{\L}'s JSPS Postdoctoral Fellowship at the University of Tokyo. M{\L} was in part supported by the Kakenhi Grant-in-Aid No.\ 21F20811. Moreover, M{\L} and PR were in part supported by the National Science Centre, Poland through the grant 2017/26/M/ST1/00700.

\end{document}